\documentclass[hidelinks, 11pt, reqno]{amsart}

\usepackage{amsmath, bm}
\usepackage[psamsfonts]{amssymb}
\usepackage{amsfonts}
\usepackage{algorithm2e}
\usepackage{graphicx}
\usepackage{verbatim}
\usepackage{dsfont}
\usepackage{tikz-cd}
\usepackage{mathrsfs}
\usepackage{color}
\usepackage[english]{babel}
\usepackage{fontenc}
\usepackage{indentfirst}
\usepackage{tikz}
\usetikzlibrary{matrix,arrows,decorations.pathmorphing}
\usepackage{algorithm2e}
\usepackage[all]{xy}
\usepackage[T1]{fontenc}
\usepackage{hyperref}
\usepackage{mathtools}
\usepackage{multicol}
\usepackage{float}
\usepackage{physics}
\usepackage{dsfont}
\usepackage{multirow}
\usetikzlibrary{tikzmark}
\usepackage{caption}
\usepackage{subcaption}
\usepackage{float}

\newtheorem{theorem}{Theorem}[section]
\newtheorem{prop}[theorem]{Proposition}
\newtheorem{lemma}[theorem]{Lemma}
\newtheorem{cor}[theorem]{Corollary}

\newtheorem{ex}[theorem]{Example}
\newtheorem{dfn}[theorem]{Definition}
\newtheorem{remark}[theorem]{Remark}

\newtheorem{thm}{Theorem}

\newcommand{\bC}{\mathbb{C}}
\newcommand{\bD}{\mathbb{D}}

\newcommand{\bK}{\mathbb{K}}

\newcommand{\bQ}{\mathbb{Q}}
\newcommand{\bR}{\mathbb{R}}
\newcommand{\bS}{\mathbb{S}}
\newcommand{\bT}{\mathbb{T}}

\newcommand{\bZ}{\mathbb{Z}}

\newcommand{\cA}{\mathcal{A}}

\newcommand{\cC}{\mathcal{C}}

\newcommand{\cH}{\mathcal{H}}

\newcommand{\cL}{\mathcal{L}}

\newcommand{\cO}{\mathcal{O}}
\newcommand{\cP}{\mathcal{P}}

\newcommand{\cS}{\mathcal{S}}

\newcommand{\ra}{\rightarrow}

\newcommand{\mL}{\mathds{1}_L}
\newcommand{\mM}{\mathds{1}_M}
\newcommand{\mW}{\mathds{1}_W}

\DeclareMathOperator{\im}{Im}
\DeclareMathOperator{\ham}{Ham}

\DeclareMathOperator{\spec}{Spec}

\DeclareMathOperator{\CZ}{CZ}
\DeclareMathOperator{\CW}{CW}
\DeclareMathOperator{\CF}{CF}
\DeclareMathOperator{\HW}{HW}
\DeclareMathOperator{\HF}{HF}
\DeclareMathOperator{\SH}{SH}
\DeclareMathOperator{\RFH}{RFH}
\DeclareMathOperator{\QH}{QH}
\DeclareMathOperator{\minmax}{minmax}
\DeclareMathOperator{\tel}{tel}

\newcommand{\qfnote}[1]{#1}

\newcommand{\jznote}[1]{#1}

\newcommand{\zjnote}[1]{#1}

\allowdisplaybreaks

\begin{document}
\title{Spectrally-large scale geometry via set-heaviness}

\author{Qi Feng}
\email{feqi@mail.ustc.edu.cn}
\address{School of Mathematical Sciences, University of Science and Technology of China, 96 Jinzhai Road, Hefei Anhui, 230026, China}

\author{Jun Zhang}
\email{jzhang4518@ustc.edu.cn}
\address{The Institute of Geometry and Physics, University of Science and Technology of China, 96 Jinzhai Road, Hefei Anhui, 230026, China}

\maketitle

\vspace{-5mm}

\begin{abstract}
We show that there exist infinite-dimensional quasi-flats in the \jznote{compactly supported} Hamiltonian diffeomorphism group of the Liouville domain, \jznote{with respect to the spectral norm}, if and only if the symplectic cohomology of this Liouville domain does not vanish. In particular, there exist infinite-dimensional quasi-flats in the \jznote{compactly supported} Hamiltonian diffeomorphism group of the unit co-disk bundle of any closed \qfnote{manifold}, which answers a question \qfnote{raised} in authors' previous work \jznote{in} \cite{FZ24}.  The similar conclusion holds for the  ${\rm Ham}$-orbit space \jznote{of} an admissible Lagrangian in any Liouville domain. Moreover, we show that if a closed symplectic manifold contains an incompressible Lagrangian with a certain topological condition, then its Hamiltonian diffeomorphism group admits infinite-dimensional flats. Proofs of all these results rely on the existence of a family of heavy hypersurfaces.
   \end{abstract}

\tableofcontents 

\section{Introduction}
Given a symplectic manifold $(M,\omega)$ (possibly with boundary), a smooth function $H\colon [0,1]\times M\ra\bR$ generates a flow $\{\varphi_H^t\}_{t \in [0,1]}$ by integrating the Hamiltonian vector field $X_{H_t}$ determined by $\omega(\cdot, X_{H_t})=dH_t$. Denote the group of time-one maps of the flow $\varphi_H^t$ where $H$ is compactly supported in $[0,1]\times{\rm int}(M)$ by $\ham(M,\omega)$, called the Hamiltonian diffeomorphism group of $(M,\omega)$. This group is endowed with several interesting metrics. For instance, 
$$\|\varphi\|_{\rm Hofer}\coloneqq \inf\left\{\int_0^1 \max_{M}H(t,\cdot)-\min_{M}H(t,\cdot)\,dt\,\bigg|\,\varphi^1_H=\varphi\right\},$$
then $d_{\rm{Hofer}}(\varphi,\psi)\coloneqq\|\varphi^{-1}\psi\|_{\rm Hofer}$ defines a Finsler-type metric on ${\rm Ham}(M, \omega)$, called the Hofer metric. 

For another instance, given a Liouville domain $(W,\omega)$, derived from (Hamiltonian) Floer theory one \qfnote{associates} spectral invariant $c(\alpha, H)$ to any pair  $(\alpha,H)\in H^*(W) \times C_c^\infty([0,1]\times W)$, which belongs to the spectrum of the action functional of $H$ (for more details, see Section \ref{ssec-LD}). For any $ \varphi_H^1 \in \ham(W,\omega)$, we define the spectral norm $\gamma$ on $\ham(W,\omega)$ by
\begin{equation} \label{dfn-spectral-norm}
\gamma(\varphi_H^1)\coloneqq -c(\mW,H)-c(\mW,\overline{H})
\end{equation}
where $\mW$ is the unit in $H^*(W)$ and $\overline{H}$ generates $(\varphi_H^1)^{-1}$. Note that $\gamma(\varphi_H^1)$ is well-defined since the spectral invariant $c(\alpha, H)=c(\alpha, F)$ for any $\alpha\in H^*(W)$  if $\varphi_H^1 = \varphi_F^1$ (see \cite{FS07}).  Then for any $\varphi,\psi\in\ham(W,\omega)$, define $d_{\gamma}(\varphi,\psi)\coloneqq\gamma(\varphi^{-1}\psi)$, which turns out to be a bi-invariant metric on $\ham(W,\omega)$.

\subsection{Liouville domain} 

For a general symplectic manifold $(M, \omega)$ (possibly with boundary), the large-scale geometry on $(\ham(M,\omega),d_{\gamma})$ has addressed much attentions in recent works (\cite{KS21,Mai24,FZ24}), which explores which unbounded metric spaces are contained inside $(\ham(M,\omega),d_{\gamma})$, up to quasi-isometry. In the setting of Liouville domains $(W, \omega)$, \cite{Mai24} confirms the unboundedness of $(\ham(W,\omega),d_{\gamma})$ under the non-vanishing assumption of $\SH^*(W,\omega)$, the symplectic cohomology of $W$ (for its background, see Section \ref{ssec-LD}). This assumption holds for instance, in coefficient $\bZ_2$ when $(W, \omega)=(D^*_gN, \omega_{\rm can})$ for some closed Riemannian manifold $(N,g)$, due to the Viterbo isomorphism in \cite{Vit99}. 

\begin{remark} \label{rmk-gamma-hofer} It is readily verified that $\gamma(\varphi) \leq \|\varphi\|_{\rm Hofer}$. Therefore, the large-scale \jznote{geometry} with respect to $d_{\rm \gamma}$ implies \jznote{the large-scale geometry} with respect to $d_{\rm Hofer}$ \jznote{which has a} rich history (see \cite{Mil01,Ush13,Ush14,PS23,CGHS24}). For more related discussion, see Section \ref{ssec-Lag-Hofer} below. \end{remark}

In this paper, we will continue to investigate more complicated large-scale phenomena in the setting of Liouville domains. We will \jznote{mainly} focus on the $\bZ_2$-coefficient, except when we discuss the issue of coefficient \jznote{changing in Section \ref{ssec-coefficient}}. \jznote{Also,} we only focus on \jznote{Liouville} domain $W$ with $2c_1(TW)|_{\pi_2(W)}=0$, \jznote{so} that the \jznote{resulting} Floer cohomology groups are $\bZ$-graded. Recall that for two metric spaces $(X, d_X)$ and $(Y, d_Y)$, a map $f: X \to Y$ is called a {\it quasi-isometric embedding} if there exist constant $A, B >0$ such that 
\begin{equation} \label{dfn-qe}
\frac{1}{A} \cdot d_X(x_1, x_2) - B \leq d_Y(f(x_1), f(x_2)) \leq A \cdot d_X(x_1, x_2) + B 
\end{equation}
for any $x_1, x_2 \in X$. If $(X, d_X)= (\bR^N, d_{\infty})$ where $N \in \bR \cup \{+\infty\}$ and $d_{\infty}$ is induced by the standard norm $\|\cdot\|_{\infty}$ in $\bR^N$, then we call $(Y, d_Y)$ {\it contains a rank-$N$ quasi-flat} if the relation (\ref{dfn-qe}) holds for some $f: (\bR^N, d_{\infty}) \to (Y, d_Y)$. \jznote{The} main result in this section is the following. 

\begin{thm}\label{thm-A}
Let $(W,\omega)$ be a Liouville domain, then  $\SH^*(W,\omega)\neq 0$ if and only if there exists a quasi-isometric embedding from $(C_{c}^\infty(I), d_{\infty})$ to $(\ham(W,\omega), d_{\gamma})$, where $(C_{c}^\infty(I), d_{\infty})$ is the metric space of compactly supported smooth functions on an open interval $I=(0,1) \subset \bR$ equipped with the $C^0$-distance. In particular, if $\SH^*(W,\omega)\neq 0$, then $(\ham(W,\omega), d_{\gamma})$ contains a rank-$\infty$ \jznote{quasi-}flat. 
\end{thm}

The last conclusion of Theorem \ref{thm-A} comes from \jznote{\text{Th\'{e}or\`{e}me} 10 in \cite{Ban93}:} every separable metric space admits an isometric embedding into $C_c^0(I)$. 

\medskip

Our method of proving Theorem \ref{thm-A} is inspired by the approach in \cite{Sun24}, which partially relies on Entov-Polterovich's theory of heavy sets (where the heaviness is defined for subsets only in a closed symplectic manifold). In this paper, we generalize this concept to non-closed case, say Liouville domains. Explicitly, instead of using symplectic quasi-states, we define heaviness in terms of the spectral invariant $c(\alpha, H)$ directly (see Definition~\ref{dfn-heavy} below). With the help of heaviness, we are able to control the behavior of the spectral invariant $c(\alpha, H)$ for those $H$ supported in a neighborhood of the (heavy) skeleton $(W, \omega)$. This serves as the key step in the proof of Theorem~\ref{thm-A}. Recall that the skeleton of $(W, \omega)$, \jznote{denoted} by ${\rm Sk}(W)$, is the subset of $W$ where all points eventually shrink to along the negative flow of the defining Liouville vector field in $(W, \omega)$ (for its formal definition, see (\ref{dfn-skeleton})). 

\medskip

As a matter of fact, Theorem \ref{thm-A} \jznote{comes from} the following richer result, where Theorem \ref{thm-A} \jznote{is the} equivalence $(1) \Longleftrightarrow (4)$. 

\begin{prop}\label{thm-Liouville}
For any Liouville domain $(W,\omega)$ and interval $I = (0,1) \subset \bR$, the following are equivalent:
\begin{enumerate}
\item $\SH^*(W,\omega)\neq 0$.
\item ${\rm Sk}(W)$ is heavy in $(W, \omega)$. 
\item $\{r_0\}\times\partial W$ is heavy in $(W, \omega)$ for any $0<r_0<1$. 
\item There exists a quasi-isometric embedding from $(C^\infty_c(I), d_{\infty})$ to $(\ham(W,\omega), d_{\gamma})$.
\end{enumerate}
\end{prop}

\begin{proof} The proof combines several results from this paper and other references. 

(1)$\Rightarrow$(2): By Lemma~C in \cite{Mai24}.

(2)$\Rightarrow$(3): By Proposition~\ref{thm-partial-heavy}.

(3)$\Rightarrow$(4): By Lemma~\ref{lem-flat-1}.

(4)$\Rightarrow$(1): By Theorem~1.3 in \cite{BK22}.
\end{proof}

\jznote{Here are some remarks about Proposition \ref{thm-Liouville}}. 

\medskip

\noindent (1) \jznote{Proposition \ref{thm-Liouville} and Theorem 1.3 in \cite{BK22}} together establish a dichotomy  that either $(\ham(W,\omega),d_{\gamma})$ is bounded or  $(\ham(W,\omega),d_{\gamma})$ admits a quasi-isometric embedding from $(C^\infty_c(I),d_{\infty})$.  This dichotomy is \jznote{distinguished} by the symplectic cohomology of $(W,\omega)$ \jznote{- whether ${\rm SH}^*(W, \omega)$ vanishes or not.}
\medskip

\noindent (2) For closed symplectic manifolds, \jznote{by the main result of \cite{MSV24}}, one can use relative symplectic cohomology (invented by Varolgunes in \cite{Var21}) to characterize heaviness \jznote{of a subset (which is called ${\rm SH}$-heaviness)}. We can regard the equivalence of (1) and (2) in Proposition \ref{thm-Liouville} as an analogue of Theorem~1.7 in \cite{MSV24} \jznote{but} in the Liouville domain setting. 

\medskip

\noindent (3) Recently, \jznote{in \qfnote{the setting of Liouville domains, }} the symplectic cohomology with support,  \jznote{as} an open analogue of relative symplectic cohomology \jznote{mentioned in point (2) above}, has been defined in \cite{SUV25}. \jznote{One} can use the argument in Theorem 3.5 in \cite{MSV24} to show that the heaviness of some compact subset $K$ implies the $\SH$-heaviness of $K$, which means the symplectic cohomology with support on $K$ does not vanish. For the converse direction, however, we require \jznote{an} open analogue of Theorem~2.24 in \cite{MSV24} to apply the argument of Theorem~3.3 in \cite{MSV24}, thereby showing that the $\SH$-heaviness of $K$ implies its heaviness.

\medskip

\noindent \qfnote{(4) There is another type of heaviness that is  defined by  the spectral invariant of symplectic cohomology with support, as introduced in \cite{OS19}.  We refer to this type as $\rho$-heaviness for simplicity, and the equivalence of $\rho$-heaviness and $\SH$-heaviness has been established in \cite{OS19}. By \jznote{transferring through} $\rho$-heaviness, we show the equivalence of heaviness and $\SH$-heaviness, where its proof is given in Appendix~\ref{app-2}. Furthermore,  this equivalence offers an alternative  proof of (1)$\Rightarrow$(2), (3) in Proposition~\ref{thm-Liouville}, and this approach can be applied to the admissible Lagrangian setting \jznote{as in Section \ref{ssec-ad-Lag}}.}

\medskip

\noindent (5) \jznote{By Theorem~13.3 in \cite{Rit13}}, $\SH^*(W,\omega)\neq 0$ is equivalent to $\RFH^*(W)\neq 0$ \jznote{for Liouville domains $(W,\omega)$}. We can characterize  the heaviness of $\{r_0\}\times\partial W$ by using Rabinowitz Floer cohomology, which can be  defined by a sequence of $\bigvee$-shaped Hamiltonians  (see Theorem~1.5 in \cite{CFO10}). 
It would be interesting to find a direct approach to establish  the  equivalence between the non-vanishing condition $\RFH^*(W)\neq 0$ and the heaviness of $\{r_0\}\times\partial W$ (see the diagram below). 
\[\begin{tikzcd}
	{\SH^*(W,\omega)\neq 0} & {\RFH^*(W)\neq 0} \\
	{{\rm Sk}(W)\text{ is heavy}} & \begin{array}{c} \{r_0\}\times \partial W\text{ is heavy}\\ \text{for any }0<r_0<1 \end{array}
	\arrow[ Leftrightarrow, from=1-1, to=1-2]
	\arrow[Leftrightarrow, from=1-1, to=2-1]
	\arrow[Leftrightarrow, dashed,  from=1-2, to=2-2]
	\arrow[Leftrightarrow,  from=2-1, to=2-2]
\end{tikzcd}\]
\jznote{For studies in this direction, see the upcoming work \cite{KZ25}.}

\medskip

The next example lists several Liouville domains where their symplectic cohomologies do not vanish (hence, by \jznote{Theorem \ref{thm-A}} each Hamiltonian diffeomorphism group contains a rank-$\infty$ quasi-flat). \jznote{In particular,} Example \ref{ex-SH} (1) covers the case $(S^2, g)$, where the method in \cite{FZ24} fails. 

\begin{ex}\label{ex-SH}
Here are some examples of Liouville domain $(W,\omega)$ with symplectic cohomology $\SH^*(W,\omega)\neq 0$:\begin{enumerate}
\item \jznote{Unit} co-disk bundle $(D^*_gN,\omega_{\rm can})$ of a (closed) manifold $(N, g)$ (due to Viterbo's isomorphism, see \cite{Abo15}).   
\item Affine log Calabi–Yau surfaces with maximal boundary (\jznote{by} Theorem~1.2 in \cite{Pas19}). Roughly speaking, they are some complex projective surfaces minus some nodal curves and points. \jznote{For} more details see Section~3.2 in \cite{Pas19}.
\item \jznote{Liouville} domain $W$ that contains a closed exact Lagrangian submanifold $L \subset (W, \omega)$ (due to Viterbo functoriality in  \cite{Vit99}).
\item A product of Liouville domains $(W_1, \omega_1) \times \cdots\times (W_n, \omega_2)$, where {\rm each} $(W_i, \omega_i)$ \jznote{has} $\SH^*(W_i,\omega_i)\neq 0$ (due to K\"unneth formula of symplectic cohomology in \cite{Oan06}). 
\item \jznote{Liouville} domain $W$ that admits a symplectic embedding from some Liouville domain $K$ with $\SH^*(K,\omega)\neq 0$ and the image is incompressible (by Proposition~\ref{prop-extend} below). Here we say $K\subset W$ is \jznote{{\rm incompressible}}  if the inclusion-induced map $\pi_1(K)\ra\pi_1(W)$ is injective. 
\item \jznote{Liouville}  domain $W$ whose contact boundary consists of  at least $2$ connected components (by Theorem~C in recent work \cite{DR24}). Note that, by definition since Liouville domain $W$ only admits {\rm positive} boundary,  in this case those multiple boundary components are all positive boundaries.  
\end{enumerate}
\end{ex}

Similarly to the closed setting, we can show that \jznote{heaviness} implies the non-displaceable (by Hamiltonian diffeomorphisms) as follows, where the proof is given in \jznote{Section \ref{proof-prop-displace}}. 

\begin{prop}\label{prop-displace}
For any Liouville domain $(W, \omega)$, if a compact subset $K$ can be displaced from itself by a Hamiltonian diffeomorphism on $(W, \omega)$, then $K$ is not heavy.
\end{prop}

It is known that any closed submanifold $N$ is always displaceable whenever $\dim N < \frac{1}{2} \dim W$ (see \cite{Gur08}). As an immediate corollary of Proposition \ref{prop-displace} and Theorem \ref{thm-A}, we have the following result. 

\begin{cor} \label{cor-inc-Lag}
If the skeleton ${\rm Sk}(W)$ of the Liouville domain $(W, \omega)$ is a closed manifold of $\dim {\rm Sk}(W)<\frac{1}{2}\dim W$, then $\SH^*(W,\omega)=0$.  In particular, the metric space $(\ham(W,\omega),d_{\gamma})$ is bounded. 
\end{cor}

\begin{remark} \label{rmk-Kang-1} \zjnote{The conclusion in Corollary \ref{cor-inc-Lag} can also be obtained in an alternative way: from the dimensional condition of ${\rm Sk}(W)$, its displaceability implies that $W$ is displaceable inside its completion $\widehat{W}$, simply by shrinking $W$ into a (displaceable) neighborhood of ${\rm Sk}(W)$ via the flow of the defining Liouville flow of $W$. Then the vanishing result ${\rm SH}^*(W, \omega) = 0$ is given by Corollary A1 in \cite{Kang-displace}.} \end{remark} 

More about the situation where symplectic cohomology vanishes, we can use Example~\ref{ex-SH} (5)  to establish a necessary condition for $\SH^*(W,\omega)=0$. Recall for any Lagrangian $L$ in symplectic manifold $(M,\omega)$,  Weinstein's neighborhood theorem states that there exists a neighborhood $U$ of $L$ that is symplectomorphic to $D_{g}^*L$ for some metric $g$. Since $\SH^*(D_g^*L, \omega_{\rm std})\neq 0$ (over $\bZ_2$), we conclude that $\SH^*(W,\omega)=0$ implies the non-existence of incompressible Lagrangians.

\begin{cor}[cf.~Proposition~D in \cite{Mai24}]\label{cor-SH}
For any Liouville domain $(W,\omega)$ with $\SH^*(W,\omega)=0$, there does {\rm not} exist \jznote{any} incompressible Lagrangian $L$. In particular, there does {\rm not} exist any Lagrangian sphere in $(W,\omega)$ when $\dim W \geq 4$.
\end{cor}

\jznote{Here} are two examples of Liouville domain $(W,\omega)$ with  $\SH^*(W,\omega)=0$ where \jznote{Corollary~\ref{cor-SH} applies}. 

\begin{ex}\label{ex-ball}
Let $({\rm B}^{2n}, \omega_{\rm std})$ denote the unit ball in $(\bR^{2n}, \omega_{\rm std})$, it is a standard fact that $\SH^*({\rm B}^{2n},\omega_{\rm std})=0$. By Corollary \ref{cor-SH}, there does {\rm not} exist any incompressible Lagrangian $L$ in $({\rm B}^{2n}, \omega_{\rm std})$. This is a weaker vision of a result from Gromov,  Theorem~0.4.A  in \cite{Gro85}, which states that there does {\rm not} exist any Lagrangian $L$ with $H^1(L;\bR)= 0$ in $(\bC^n,\omega_{\rm std})$. \jznote{In general, there are manifolds $L$ with $\pi_1(L) \neq 0$ but $H^1(L; \bZ) = 0$.}

 More generally, let $(W,\omega)$ be a subcritical Stein domain, then $\SH^*(W,\omega)=0$ by Theorem~4.8 in \cite{Wen}.  
Alternatively, let $\partial W$ be displaceable in the completion $\widehat{W}$, then $\SH^*(W,\omega)=0$ by Theorem~13.4 in \cite{Rit13}.  Again, by Corollary \ref{cor-SH} there does {\rm not} exist an incompressible Lagrangian $L$ in $(W,\omega)$.

\end{ex}
\begin{ex}\label{ex-ball-2}
Let $(W, \omega)$ be a Liouville domain of dimension $2n \geq 4$ such that $\partial W$ is contact isomorphic to the standard $S^{2n-1}$. Then by Corollary~6.5 in \cite{Sei08}, we have $\SH^*(W,\omega) = 0$ and $c_1(TW)=0$. \jznote{Therefore, by Corollary \ref{cor-SH}}, there does not exist an incompressible Lagrangian $L$ in $(W,\omega)$. \end{ex}

\begin{remark} Related to both Example \ref{ex-ball} and Example \ref{ex-ball-2}, by Seidel and Smith in \cite{SS05}, there exists Liouville domain that is diffeomorphic to ${\rm B}^8$ but with non-vanishing symplectic homology. Therefore, by Theorem \ref{thm-A}, large-scale geometric phenomenon appears in this case (in a contrast to Example \ref{ex-ball}). In a different direction, Eliashberg, Floer and McDuff shows that in the setting of Example \ref{ex-ball-2}, $W$ is diffeomorphic to ${\rm B}^{2n}$ in \cite{McD91}, and a famous conclusion follows that Seidel-Smith's example provides a Liouville domain, with diffeomorphism type being ${\rm B^8}$, but the contact type boundary is \jznote{\rm not} contactomorphic to the standard $S^7$.
\end{remark}

\subsection{Admissible Lagrangian} \label{ssec-ad-Lag}
Instead of the Hamiltonian diffeomorphism group as discussed above, a ``relative'' situation \jznote{can be} considered. Concretely, given a Lagrangian $L$, consider the orbit space $\cO(L)$ of $L$ under the action \jznote{of} Hamiltonian diffeomorphism group $(\ham(M,\omega))$,  defined as
\[
\cO(L)\coloneqq\{\varphi(L)\mid \varphi\in\ham(M,\omega)\}.
\]
Similarly as above, for a Liouville domain $(W, \omega)$, given an admissible Lagrangian $L$ in $(W, \omega)$ (in the sense that it intersects  well with the boundary $\partial W$, see Definition~\ref{admissible-Lag}),  one can associate spectral invariant $\ell(\alpha,H)$ to any pair $(\alpha, H)\in H^*(L)\times C^\infty_c([0,1]\times W)$, \jznote{which is} derived from (filtered) wrapped Floer cohomology. Then define the spectral norm $\gamma$ on $\cO(L)$ by 
\begin{equation} \label{gamma-Lag}
\gamma(\varphi_H^1(L))\coloneqq-\ell(\mL,H)-\ell(\mL,\overline{H})
\end{equation}
where $\mL$ is the unit in $H^*(L)$.  Note that in this case $\gamma(\varphi_H^1(L))$ is well-defined, since the spectral invariant $\ell(\alpha,H)=\ell(\alpha, F)$ for any $\alpha\in H^*(L)$ if $\varphi_H^1(L)=\varphi_F^1(L)$ (see \cite{Gon24}). Then for any $\varphi(L), \psi(L)\in\cO(L)$, define $\delta_{\gamma}(\varphi(L),\psi(L))\coloneqq\gamma(\varphi^{-1}\psi(L))$.

Recall that Gong shows that the wrapped Floer cohomology $\HW^*(L)\neq 0$ if and only if $(\cO(L),\delta_{\gamma})$ is unbounded in \cite{Gon24}, where the wrapped Floer cohomology is defined by the direct limit of a sequence of $\HW^*(L; H_i)$ similar to the symplectic cohomology. A previous work from authors \cite{FZ24} shows that for some Riemannian manifolds $(N,g)$ containing $(S^{n\neq 2}, g_{\rm std})$, the orbit space of cotangent fiber $(\cO(F_q), \delta_{\gamma})$ admits a rank-$\infty$ quasi-flat, where $q\in N$ and $F_q$ is the cotangent fiber in $D^*_gN$. A relative version of Theorem \ref{thm-A} goes as follows, 
\begin{thm}\label{thm-B}
Let $L$ be an admissible Lagrangian submanifold in a Liouville domain $(W,\omega)$, then   $\HW^*(L)\neq 0$ if and only if there exists a quasi-isometric embedding from $(C^\infty_c(I), d_{\infty})$ to $(\cO(L),\delta_{\gamma})$, where $(C^\infty_c(I), d_{\infty})$  is the metric space of compactly supported smooth functions on an open interval $I = (0,1) \subset \bR$ equipped with the $C^0$-distance. In particular, if $\HW^*(L)\neq 0$, then  $(\cO(L),\delta_{\gamma})$ contains a rank-$\infty$ flat.
\end{thm}
Similarly, to prove Theorem~\ref{thm-B},  we need analogue result as Proposition~\ref{thm-Liouville}, where \jznote{the set-heaviness with respect to a given Lagrangian is needed}. Note that in the closed symplectic manifold setting, via the Lagrangian spectral invariant (\cite{LZ18}), one can define such a ``relative'' heaviness (which was first proposed and studied by Kawasaki in \cite{Kaw18}). We define  $L$-heaviness in Definition~\ref{def-L-heavy} and then Theorem \ref{thm-B} immediately comes from the following richer result, where Theorem \ref{thm-B} is \jznote{the equivalence} $(1) \Longleftrightarrow (4)$. 
\begin{prop}\label{thm-ad-Lag}
For any admissible Lagrangian \jznote{in Liouville} domain $(W,\omega)$ and interval $I = (0,1) \subset \bR$, the following are equivalent:
\begin{enumerate}
\item $\HW^*(L)\neq 0$.
\item ${\rm Sk}(W)$ is $L$-heavy in $(W, \omega)$. 
\item $\{r_0\}\times\partial W$ is $L$-heavy in $(W, \omega)$ for any $0<r_0<1$.
\item There exists a quasi-isometric embedding from $(C^\infty_c(I), d_{\infty})$ to $(\cO(L), d_{\gamma})$.
\end{enumerate}
\end{prop}

\begin{proof}
The proof combines several results from this paper and \cite{Gon24}.

(1)$\Rightarrow$(2): By the proof of Theorem~7 in \cite{Gon24}, which is an open analogue of Lemma~C in \cite{Mai24}.

(2)$\Rightarrow$(3): By Proposition~\ref{thm-Lag-partial-heavy}.

(3)$\Rightarrow$(4): By Lemma~\ref{lem-Lag-flat}.

(4)$\Rightarrow$(1): By Theorem~6 in \cite{Gon24}.
\end{proof}

\jznote{Next} example establishes a family of admissible Lagrangians where their wrapped Floer cohomologies do not vanish. Hence, by \jznote{Theorem~\ref{thm-B}}, each orbit space of the admissible Lagrangian contains a rank-$\infty$ quasi-flat. In particular, Example~\ref{ex-HW} below covers the case that $L$ is a cotangent fiber in $D^*_gS^2$, where the method in \cite{FZ24} fails.

\begin{ex}\label{ex-HW}
\zjnote{Here is a standard example of an admissible Lagrangian $L$ in a Liouville domain $(W,\omega)$ with  $\HW^*(L)\neq 0$: any disk conormal bundle $L = \nu^*K$ of a closed submanifold $K\subset N$ in the unit co-disk bundle $(W, \omega) = (D_g^*N,\omega_{\rm can})$ of a closed manifold $(N,g)$. This is due to Viterbo's isomorphism (see \cite{APS08,AS10}). \jznote{In particular, if one takes} $K=\{q\}$, then $\nu^*K=F_q$\jznote{, a} cotangent fiber.}
\end{ex}

For any admissible Lagrangian $L$ with $\HW^*(L)\neq 0$, we can use $L$-heaviness and $L$-superheaviness to show that the skeleton ${\rm Sk}(W)$ can not be displaced from $L$. Here, $L$-superheaviness serves as an open analogue to superheaviness and is  defined in Definition~\ref{def-L-heavy}.

\begin{prop}\label{prop-Lag-heavy}
For any admissible Lagrangian $L$ in a Liouville domain $(W,\omega)$, $L$ is $L$-superheavy.
\end{prop}
\begin{proof}
By \jznote{the} homotopy invariance of the spectral invariant.
\end{proof}
\begin{cor}\label{cor-Lag-displace}
If $\HW^*(L)\neq 0$, then $\phi(L)\cap {\rm Sk}(W)\neq \emptyset$ and $\phi(L)\cap\{r_0\}\times\partial W\neq \emptyset$ for any $\phi\in\ham(W,\omega)$ and any $0<r_0<1$. 
\end{cor}
\begin{proof}
Since $\HW^*(L)\neq 0$ implies $\HW^*(\phi(L))\neq 0$ for any $\phi\in\ham(W,\omega)$, then ${\rm Sk}(W)$ and $\{r_0\}\times\partial W$ are $\phi(L)$-heavy in $(W,\omega)$ for any $0<r_0<1$. By Proposition~\ref{prop-Lag-heavy} and Proposition~\ref{prop-superheavy},  we have $\phi(L)\cap{\rm Sk}(W)\neq 0$.
\end{proof}
We can use Corollary~\ref{cor-Lag-displace} to show $\HW^*(L)=0$ \jznote{under the hypothesis that one} can displace the skeleton ${\rm Sk}(W)$ from $L$. \jznote{The following is an example}.
\begin{ex}
Let $(W,\omega)=({\rm B}^{2n},\omega_{\rm std})$ and $L=\{y=0\}\cap {\rm B}^{2n}$, then $\HW^*(L)=0$. \jznote{One can also obtain this result by the ${\rm SH}$-module structure of ${\rm HW}$ as Theorem 10.6 in \cite{Rit13}}. 
\end{ex}

\begin{remark} \label{rmk-Kang-2} \zjnote{The conclusion of Corollary \ref{cor-Lag-displace} has been proved in Theorem 9.11(b) in \cite{CO18} (see also the first paragraph in Section 7 in \cite{CDRGG24}), which follows the same idea as in \cite{Kang-displace} (cf.~Remark \ref{rmk-Kang-1}). } \end{remark}

\subsection{Closed manifold}
For closed weakly monotone (or semi-positive) symplectic manifolds, \jznote{one defines the spectral invariant using the PSS map and it has been well-studied a few decades ago (see \cite{Sch00, Oh05})}. However, in general the homotopy invariance property of spectral invariant now holds only for the universal cover $\widetilde{\ham}(M,\omega)$ of $\ham(M,\omega)$ (instead of $\ham(M,\omega)$). The spectral pseudo-norm is defined as follows: 
\[
\gamma(H)\coloneqq-c(\mM,H)-c(\mM,\overline{H})
\]
which gives a well-defined map on $\widetilde{\ham}(M,\omega)$:
\[
\widetilde{\gamma}\colon\widetilde{\ham}(M,\omega)\ra\bR_{\geq 0},\quad\widetilde{\gamma}([\{\varphi_H^t\}_{t\in[0,1]}])\coloneqq\gamma(H).
\]
By taking the infimum over all possible Hamiltonians, we obtain the spectral norm:
\[
\gamma\colon\ham(M,\omega)\ra\bR_{\geq 0},\quad\gamma(\varphi)\coloneqq\inf_{\varphi_H^1=\varphi}\gamma(H).
\]

By using a similar approach based on the set-heaviness that was employed in the proof of Theorem~\ref{thm-A},  along with the result from estimating the boundary depth, we can prove a similar result in a closed setting:

\begin{thm}\label{thm-C}
Let $(M,\omega)$ be a closed weakly monotone symplectic manifold. If $(M, \omega)$ contains an incompressible Lagrangian that admits a Riemannian metric of nonpositive sectional curvature and if the pseudo-norm $\widetilde{\gamma}$ descends to the norm $\gamma$, then there exists a quasi-isometric embedding from $(C_{c}^\infty(I), d_{\infty})$ to $(\ham(M,\omega), d_{\gamma})$. In particular, under the same hypothesis, $(\ham(M,\omega), d_{\gamma})$ contains a rank-$\infty$ quasi-flat. 
\end{thm}

The above theorem is derived from the following \jznote{proposition}, which generalizes Theorem~D in \cite{KS21} \jznote{and its proof is given in Section \ref{ssec-proof-thm-closed}}. The original version does not  require  the pseudo-norm $\widetilde{\gamma}$ to descend to the norm $\gamma$ and it only establishes the existence of a rank-$\infty$ quasi-flat.
\begin{prop}\label{thm-closed}
Let $(M,\omega)$ be a closed weakly monotone symplectic manifold that admits a nonconstant autonomous Hamiltonian $H\colon M\ra\bR$ such that all contractible closed orbits of $X_H$ are constant.  If the pseudo-norm $\widetilde{\gamma}$ descends to the norm $\gamma$, then there exists a quasi-isometric embedding from $(C_{c}^\infty(I), d_{\infty})$ to $(\ham(M,\omega), d_{\gamma})$. In particular, under the same hypothesis, $(\ham(M,\omega), d_{\gamma})$ contains a rank-$\infty$ quasi-flat. 
\end{prop}
\begin{proof}[Proof of Theorem~\ref{thm-C} \jznote{(Assuming Proposition \ref{thm-closed})}]
Let $L$ be the incompressible Lagrangian and $g$ the Riemannian metric of nonpositive sectional curvature. Up to rescaling, there exists a Weinstein neighborhood $D^*_gL$ of $L$ is in $M$. We identify $D^*_gL$ with $\{0\leq r\leq 1\}\times S^*_gL$. Consider the smooth function $H\colon M\ra\bR$ defined by the following properties:
\begin{enumerate}
\item $H\equiv 1$  on $M\backslash D^*_gL$ and $H\equiv 0$ on $L$.
\item $H$ only depends on $r$ in the collar region $\{0\leq r\leq1\}\times S^*_gL$.
\end{enumerate}
Then the  closed orbits of $X_H$ correspond to some closed geodesics of $L$, which are not contractible. Then $H$ satisfies the hypothesis in Proposition~\ref{thm-closed}. 
\end{proof}

As a comparison, as a direct application of Theorem~3.4 in \cite{Sun24} we obtain the following result, which generalizes the main result of \cite{Sun24} by removing the dimension condition on Lagrangians. Recall that $(M,\omega)$ is a {\it negatively monotone}  manifold if $c_1(TM)|_{\pi_2(M)}=\lambda\cdot\omega|_{\pi_2(M)}$ for some negative constant $\lambda$.  We say that $(M,\omega)$ is a {\it symplectic Calabi-Yau} manifold if $c_1(TM)|_{\pi_2(M)}=0$. Additionally, we say $(M,\omega)$ is {\it integral} if $\{\omega(B)\mid B\in\pi_2(M)\}\subset \bR$ is a discrete subgroup.

\begin{prop}[cf.~Theorem~1.1 and Corollary~1.9 in \cite{Sun24}]\label{thm-sun}
Let $(M,\omega)$ be a closed negatively monotone symplectic manifold or a closed  integral symplectic Calabi-Yau manifold on which the pseudo-norm $\widetilde{\gamma}$ descends to the norm $\gamma$. If $M$ contains an incompressible heavy Lagrangian, then there exists a quasi-isometric embedding from $(C_{c}^\infty(I), d_{\infty})$ to $(\ham(M,\omega), d_{\gamma})$. In particular, under the same hypothesis, the metric space $(\ham(M,\omega), d_{\gamma})$ contains a rank-$\infty$ quasi-flat. 
\end{prop}

\begin{ex}[Example~1.2, 1.3 and 1.4 in \cite{Sun24}]\label{ex-sum}
Here, we list some closed symplectic manifolds that satisfy the hypothesis in Proposition~\ref{thm-sun}:
\begin{enumerate}
\item Any closed symplectically aspherical $(M,\omega)$ contains an incompressible Lagrangian.
\item Any degree-$d$ smooth hypersurfaces in $\bC P^{n+1}$ with $d>n+2$. 
\item Any closed symplectic manifold of the form $(M\times M, -\omega\times\omega)$, where $(M,\omega)$ is a closed negatively monotone symplectic manifold. 
\end{enumerate}
\end{ex}

Note that, compared with Proposition~\ref{thm-sun}, Theorem \ref{thm-C} has less constraint on the ambient manifold $(M, \omega)$ but more geometric condition on the embedded Lagrangian submanifold. A standard example that satisfies the hypothesis of Theorem \ref{thm-C} is $\bT^{2n}$.  It would be interesting to find an example that covered by Theorem \ref{thm-C} but not \jznote{by} Proposition~\ref{thm-sun}.  For a potential example (if exists), one can consider the product $(M\times M,-\omega\times\omega)$, where $M$ is a weakly monotone symplectic manifold admits a Riemannian metric of nonpositive sectional curvature and the minimal Chern number $N\geq n+1$  and $\QH_*(M,\omega)$ is undeformed  (see Example~\ref{ex-descend} (3)).

\medskip

\jznote{About the} descendent property of $\gamma$ that appears in both Theorem \ref{thm-C} and Proposition \ref{thm-sun}, the following are some examples where the spectral invariant $c$ descends from $\widetilde{\rm Ham}$ to $\ham$:

\begin{ex}\label{ex-descend}
McDuff (Theorem~1.3 and Proposition 3.1 in \cite{McD10}) and Kawamoto (Theorem~6 in \cite{Kaw22}) present  examples that the spectral invariant $c$ descends to $\ham$. We only list three simple cases: if $(M,\omega)$ is a closed symplectic $2n$-dimensional manifold with one of the following conditions:
\begin{enumerate}
\item  \jznote{the symplectic structure} $\omega$ vanishes on $\pi_2(M)$;
\item \jznote{$(M,\omega)$ is negatively monotone};
\item  The minimal Chern number $N$ satisfies $N\geq n+1$, and either $H_{\rm ev}(M)$ is generated as a ring by the divisors $H_{2n-2}(M)$ or that $\QH_*(M,\omega)$ is undeformed, except in the case where $N\leq 2n$ and $(M,\omega)$ is strongly uniruled. Here $\QH_*(M,\omega)$ being undeformed means the quantum product reduces to the usual intersection \jznote{product;} $(M,\omega)$ being strongly uniruled means some genus zero Gromov–Witten invariant of the form $\left<{\rm pt},a_2, a_3\right>_{\beta}$ does \jznote{\rm not} vanish \jznote{for some class $\beta\neq 0$}. Every Calabi–Yau 3-folds satisfy this condition as any K\"ahler 6-manifold $M$ satisfies that the even degree homology is generated as a ring by $H_4(M)$. Every \jznote{hyperk\"ahler} manifolds satisfy this condition as  \jznote{hyperk\"ahler} manifolds have undeformed quantum products.
\end{enumerate}
\end{ex}

As symplectic manifolds are orientable, any compact $2$-dimensional symplectic manifold (possibly with boundary) is diffeomorphic to some $\Sigma_g\backslash n\bD^2$, where $\Sigma_g\backslash n\bD^2$ represents a surface of genus $g\geq 0$ with $n$ disks removed. If $M$ is a closed surface not diffeomorphic to $S^2$ or a compact surface with boundary not diffeomorphic to $\bD^2$, then $M$ contains an incompressible Lagrangian circle. Combine Theorem~\ref{thm-A} and \ref{thm-C}, we obtain the following result:
\begin{cor}
Let  $(M,\omega)$ be a $2$-dimensional symplectic manifold possible with boundary. If $M$ is not diffeomorphic to $S^2$ or $\bD^2$, then there exists a quasi-isometric embedding from $(C_{c}^\infty(I), d_{\infty})$ to $(\ham(M,\omega), d_{\gamma})$. 
\end{cor}

\begin{remark} \label{rmk-s2-disk} 
For the case of $S^2$, the metric space $(\ham(S^2,\omega_{\rm std}),d_{\gamma})$ is bounded (see \cite{EP03}) and $(\ham(S^2,\omega_{\rm std}), d_{\rm Hofer})$ admits an quasi-isometric embedding from $(C^\infty_c(I),d_{\infty})$ (see \cite{PS23}). Similarly, for the case of $\bD^2$, $(\ham(\bD^2,\omega_{\rm std}), d_{\gamma})$ is bounded by Theorem~\ref{thm-A} (in fact, by the half direction confirmed earlier by \cite{BK22}) and $(\ham(\bD^2,\omega_{\rm std}), d_{\rm Hofer})$ admits an quasi-isometric embedding from $(C^\infty_c(I), d_{\infty})$ \jznote{by \cite{Sey14}}. In fact, when $(M, \omega)$ is a {\rm closed} surface (in particular, without boundary) and with genus at least $1$, \jznote{Theorem~D in \cite{KS21} and Usher's estimation of the boundary depth in \cite{Ush13} together} show that $(\ham(M,\omega), d_{\gamma})$ admits  rank-$\infty$ quasi-flats. 
\end{remark}


If we are only seeking for the existence of a quasi-flats instead of a rank-$\infty$ quasi-flats, we can consider heavy sets directly.
\begin{thm}\label{thm-D}
Let $(M,\omega)$ be a closed symplectic manifold on which the pseudo-norm $\widetilde{\gamma}$ descends to the norm $\gamma$. If $M$ contains a compact subset $K$ which is heavy but not superheavy, then $(\ham(M,\omega), d_{\gamma})$ contains a rank-$1$ quasi-flat. Moreover, if there \jznote{exist} (at least two) pairwise disjoint heavy compact subsets $K_0,\cdots,K_n$ with $\overline{K_i}\cap\overline{K_j}=\emptyset$ for any $i\neq j$, then $(\ham(M,\omega), d_{\gamma})$ contains a rank-$n$ quasi-flat. 
\end{thm}
\begin{remark}
We can view Lemma~\ref{lem-flat-1}, \ref{lem-Lag-flat} and \ref{lem-flat-2} as \jznote{infinite-dimensional} versions of Theorem~\ref{thm-D}.
\end{remark}

\subsection{Summary of current art}
\jznote{In this section,} we summarize the approaches to obtain the large-scale geometry with respect to  the spectral norm, \jznote{elaborated on $\bT^2$ as a concrete example}. \jznote{Most results on such large-scale geometry are obtained through two main approaches.} 
These \jznote{two} approaches work for more general settings, including higher genus surfaces, higher-dimensional tori $\bT^{2n}$, Example~\ref{ex-sum}, Liouville domains with non-vanishing symplectic cohomology \jznote{(as in Theorem \ref{thm-A})}, and the $\ham$-orbit space of admissible Lagrangians with non-vanishing wrapped Floer cohomology in Liouville domains \jznote{(as in Theorem \ref{thm-B})}.

\medskip


For any $f\in C^\infty_c([0,1])$, we can associate an Hamiltonian $\widetilde{f}\in C^\infty_c(\bT^2)$ to $f$ that is compactly supported on a cylinder $[0,1]\times S^1$ and depends only on the first coordinate $s$, as illustrated in Figure~\ref{Fig-fa}.  We will identify $f$ with a function on $S^1$. 
\begin{figure}[ht]
	\centering
\includegraphics[scale=0.65]{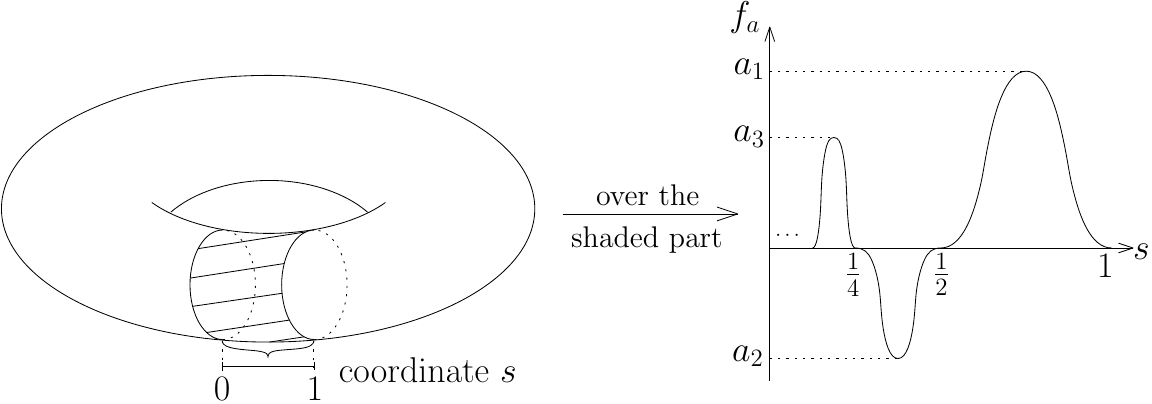}
\caption{The function $f_a$ \jznote{associated to $a=(a_i)_{i=1}^\infty\in\bR^\infty$.}}\label{Fig-fa}
\end{figure}

\noindent{\jznote{Approach one:}} \jznote{A combination of Usher's} large-scale results on the boundary depth \jznote{(Theorem~5.6 in \cite{Ush13}) and} \jznote{Kislev-Shelukhin's} inequality between the boundary depth and the spectral norm \jznote{(Theorem~A in \cite{KS21})}. More precisely, \jznote{given any $a=(a_i)_{i=1}^\infty\in\bR^\infty$, define function $f_a\in C^\infty_c([0,1])$ as illustrated in the Figure~\ref{Fig-fa}. By calculations on boundary depth $\beta$ in \cite{Ush13}, we have $\beta(\widetilde{f_a})\geq -\min_i a_i$. Consider the embedding $\Psi\colon\bR^\infty\ra\ham(\bT^2,\omega_{\rm std})$  defined by
\begin{equation} \label{dfn-fa}
\Psi(a)\coloneqq\phi_{\widetilde{f_a}}^1.
\end{equation}
Then $\beta(\widetilde{f_{b-a}})=\beta(\widetilde{f_{a-b}}) \geq \|a-b\|_{\infty}$, which implies that for any $a,b \in\bR^\infty$, we have
\[
d_{\gamma}(\Psi(a),\Psi(b))=\gamma(\phi_{\widetilde{f_a}}^{-1}\phi_{\widetilde{f_b}}^1)=\gamma(\widetilde{f_{b-a}})\geq \beta(\widetilde{f_{b-a}})\geq \|a-b\|_{\infty}.
\]
where the first inequality is from \cite{KS21}. This serves as the key step in the proof that the embedding $\Psi\colon(\bR^\infty,d_{\infty})\ra(\ham(\bT^2,\omega_{\rm std}),d_{\gamma})$ as in (\ref{dfn-fa}) is a quasi-isometric embedding.} 

\medskip

\noindent{\jznote{Approach two:}} \jznote{Show the existence of a family of disjoint heavy sets}. In $(\bT^2,\omega_{\rm can})$, the Lagrangian $L$ in the form of $\{x_0\}\times S^1$ is heavy for any $x_0\in S^1$. For any $f\in C^\infty_c([0,1])$, we can calculate the spectral norm of $\widetilde{f}$ directly. Let $f(x_1)=\max_{s\in[0,1]} f(s)$ and $f(x_2)=\min_{s\in[0,1]} f(s)$ for some $x_1,x_2\in S^1$, then we have  $c(\mathds{1}_{\bT^2}, f(x_1)-\widetilde{f})=0$ and $c(\mathds{1}_{\bT^2},\widetilde{f}-f(x_2))=0$ since $\{x_1\}\times S^1$ and $\{x_2\}\times S^1$ are heavy. The spectral norm $\gamma(\widetilde{f})$ is
\[
\gamma(\widetilde{f})=-c(\mathds{1}_{\bT^2},\widetilde{f})-c(\mathds{1}_{\bT^2}, -\widetilde{f})=-\widetilde{f}(x_2)-(-\widetilde{f}(x_1))=\max_{s\in[0,1]} f(s)-\min_{s\in[0,1]}f(s).
\]
Then the embedding $\Psi\colon (C^\infty_c([0,1]),d_{\infty})\ra(\ham(\bT^2,\omega_{\rm std}),d_{\gamma})$ defined by $\Psi(f)\coloneqq\phi_{\widetilde{f}}^1$ is a quasi-isometric embedding.
We illustrate the function used in Figure~\ref{Fig-heavy}.
\begin{figure}[ht]	
\centering\includegraphics[scale=1]{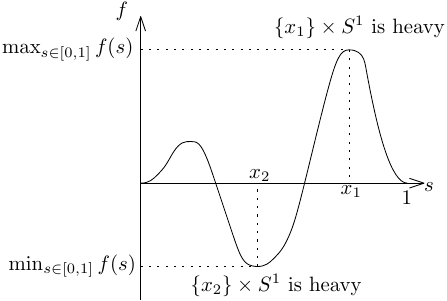}
\caption{$\widetilde{f}^{-1}(\max_{s\in[0,1]}f)$ and $\widetilde{f}^{-1}(\min_{s\in[0,1]}f)$ are heavy}\label{Fig-heavy}
\end{figure} 

\begin{remark} In general, the first approach yields Theorem~D in \cite{KS21}, which shows the existence of a rank-$\infty$ quasi-flat without requiring $\widetilde{\gamma}$ to descend to $\gamma$. The second approach leads to Theorem~\ref{thm-A}, \ref{thm-B}, Proposition~\ref{thm-closed} and \ref{thm-sun}, which \jznote{shows} the existence of quasi-isometric embedding from the function space $C^\infty_c([0,1])$ but do require $\widetilde{\gamma}$ to descend to $\gamma$. \end{remark}

Additionally, there are some other approaches to show the unboundedness of the metric space $(\ham(M,\omega), d_{\gamma})$. For instance, the egg-beater model, which is constructed by time-dependent Hamiltonians,  can be employed to show quasi-isometric embeddings from free groups, see \cite{PS16, AKKK19, Zha19}. In fact, these works establish the large-scale results concerning boundary depth. \jznote{Therefore, in the spirit of Approach one}, we can also obtain the large-scale results on the spectral norm \jznote{from this model}. \jznote{Compared with the $\bT^2$ example in Approach one, arguments based on egg-beater model often involves} time-dependent Hamiltonians.  Furthermore, we can also use the quasi-morphisms defined by the homogenization of the spectral invariant to show the unboundedness of  $(\ham(M,\omega), d_{\gamma})$, more details see \cite{KS24}.

\subsection{Coefficients} \label{ssec-coefficient}
Although we are primarily working with $\bZ_2$-coefficient, changing of coefficient will potentially influence the following two factors in our story: 
\begin{itemize}
\item[(i)] the spectral invariant $c(\mW,H)$, in its full notation $c(\mW,H; R)$, which is defined based on the Hamiltonian Floer cohomology group $\HF^*(H; R)$ in $R$-coefficient (via $PSS$-map); 
\item[(ii)] the symplectic cohomology $\SH^*(W,\omega;R)$ which is also defined \jznote{via} Hamiltonian Floer cohomologies in $R$-coefficient (see more explanations right below). 
\end{itemize}
As a consequence, spectral norm in (\ref{dfn-spectral-norm}) potential relies on the coefficient $R$. 

More explicitly, symplectic cohomology can be defined in more general cases, including other  abelian groups or even local systems (see \cite{Abo15,AFO17}).  If we replace $\bZ_2$ with  other abelian group, Proposition~\ref{thm-Liouville} still holds, as the argument does not depend on the choice of the coefficients. However, $(\ham(D^*_gN,\omega_{\rm can}), d_{\gamma})$ may be bounded, as the symplectic cohomology of $D^*_gN$ may vanish (over different coefficients). 

Seidel showes that the symplectic cohomology of $T^*\bC P^2$ (or its unit co-disk bundle) will vanish if the coefficient field $\bK$ is a field  of characteristic ${\rm char}(\bK)\neq 2$ in his unpublished note \cite{Sei10}. More rigorously, in \cite{BKK24} it proves that the symplectic cohomology of $T^*\bC P^2$ vanishes with $\bQ$-coefficients.  In this case, by Proposition~\ref{thm-Liouville}, $(\ham(D^*_g\bC P^2,\omega_{\rm can}), d_{\gamma})$ is bounded in any field $\bK$ with characteristic ${\rm char} (\bK)\neq 2$.

Here is another interesting observation. Given a ring  homomorphism $j\colon R\ra R'$, it induces a morphism $j\colon H^*(W,R)\ra H^*(W,R')$ on cohomology groups, which furthermore implies that 
\begin{equation} \label{spec-coeff}
c(\mW, H; R)\leq c(j(\mW), H; R')
\end{equation}
for every Hamiltonian $H$. For the proof of (\ref{spec-coeff}), see a homological version in Proposition 3.1.5 of \cite{KS24}. The $\bZ$-coefficient is a special case, as the spectral norm in $\bZ$-coefficient provides an upper bound for other coefficients. 

Since we have shown $(\ham(D^*_g\bC P^2,\omega_{\rm can}), d_{\gamma})$ is unbounded in $\bZ_2$-coefficient, it follows that $(\ham(D^*_g\bC P^2,\omega_{\rm can}), d_{\gamma})$ is also unbounded in $\bZ$-coefficient. Consequently, for $\bZ$-coefficient, $\SH^*(D^*_g \bC P^2,\omega_{\rm can};\bZ)$ is non-zero by Proposition~\ref{thm-Liouville}.  If ${\rm char}(\bK)\neq 2$, then $\SH^*(D^*_g\bC P^2,\omega_{\rm can}; \bK)=0$, which implies that  $\SH^*(D^*_g\bC P^2,\omega_{\rm can};\bZ)$ is $2$-torsion.

\begin{remark} In \cite{KS24}, Kawamoto and Shelukhin show that on $\bC P^n$, the spectral norm over $\bZ$-coefficient is unbounded (cf.~Remark \ref{rmk-s2-disk}). \end{remark}

\subsection{Hofer norm in Lagrangian setting} \label{ssec-Lag-Hofer} 
To end this introduction, let us get back to the Hofer geometry but in the relative setting as in the previous section. Recall that the Hofer norm $\|\cdot\|_{\rm Hofer}$ on $\ham(M,\omega)$ induces a genuine metric (called the Chekanov-Hofer metric) defined as follows, for any $L_1, L_2\in\cO(L)$,
\[
\delta_{\rm CH}(L_1,L_2)\coloneqq\inf\{\|\varphi\|_{\rm Hofer}\mid \varphi(L_1)=L_2, \varphi\in\ham(M,\omega) \}.
\]
In both Liouville domain and closed settings, results on large-scale geometry under $d_{\rm CH}$ is rarely known. Coherent to Remark \ref{rmk-gamma-hofer}, since the spectral norm always serves as a lower bound of the Hofer norm (in both absolute and relative situations), one \jznote{can approach} to this question from the study of Lagrangian spectral norm $\delta_{\gamma}$  above. 

\medskip


Here are some known cases about $(\cO(L), \delta_{\rm CH})$ from the literature.

\medskip

\noindent (1) Let $L$ be the unit circle inside the plane $(\bR^2,\omega_{\rm std})$, then the diameter of $(\cO(L), \delta_{\rm CH})$ is no larger than $2\pi$ (see Section~9 in \cite{Ush13}). 

\medskip

\noindent (2) Khanevsky proves this \jznote{for $M=\bD^2$ and $L$ the diameter}, as well as \jznote{for} $M=S^1\times (-1,1)$ and $L=S^1\times\{0\}$, one has $(\cO(L),\delta_{\rm CH})$ is unbounded (see \cite{Kha09}).

\medskip

\noindent (3) Usher proves that if $L = S^1 \times \{\rm pt\} \subset (\bT^2, \omega_{\rm std})$, where $S^1$ is a circle of $\bT^2 = S^1 \times S^1$, then $(\cO(L),\delta_{\rm CH})$ contains a rank-$\infty$ quasi-flat (see \cite{Ush13}).  

\medskip

\noindent (4) Seyfaddini generalizes Khanevsky’s first result to the case $M={\rm B}^{2n}$ and $L=\{y=0\}\cap {\rm B}^{2n}$. In fact,  Seyfaddini shows there exists a quasi-isometric embedding from $(C^\infty_c(I),d_{\infty})$ to $(\cO(L),\delta_{\rm CH})$ (see \cite{Sey14}). 

\medskip

\noindent (5) Recently, Trifa shows in the case where $M=\bD^2$ and $L$ is  a disjoint union of $k$ embedded smooth closed simple curves bounding discs of the same area $A>\frac{1}{k+1}$, one has $(\cO(L),\delta_{\rm CH})$ is unbounded (see \cite{Tri24}). 

\medskip

Interestingly, by \cite{KS21,She22b, Gon24}, case  (2), (3), (4) above \jznote{become} bounded if we replace $\delta_{\rm CH}$ by Lagrangian spectral norm $\delta_{\gamma}$. Here, $\delta_{\gamma}$ is defined directly on the orbit $\mathcal O(L)$ by taking all necessary infimums. For (1), there does not exist \jznote{non-trivial} spectral invariant to define the spectral norm, \jznote{since $L$ is displaceable}. For (5), \jznote{the (pseudo)-metric $\delta_{\gamma}$ could be defined but its properties have not been fully investigated, since quantitative theory on the Lagrangian Floer (co)homology for disjoint union of Lagrangians (or Lagrangian orbifolds) is still under construction \cite{Che21, CGHMSS22, Che22, PS23, MT23, MSS25}.} 

\medskip
  
Following the \jznote{theme of results} in this paper, we generalize Trifa's result in (5) right above in the following \jznote{theorem}, where its proof is given in Appendix~\ref{app}. 

\begin{thm}\label{thm-flat}
Let $(\bD, \omega_{\rm std})$ be a standard disk in $\bC$ with area 1, and $\underline{L}$ be a disjoint union of $k(\geq 2)$ embedded smooth closed simple curves bounding discs of the same area $A$ in $(\bD, \omega_{\rm std})$ such that $A>\frac{1}{k+1}$, then exists a quasi-isometric embedding from $(C^\infty_c(I),d_{\infty})$ to $(\cO(\underline{L}), \delta_{\rm CH})$.
\end{thm}


\subsection*{Acknowledgements}
\jznote{We thank useful conversations with Jungsoo Kang, Sungho Kim, Ibrahim Trifa, and Zhengyi Zhou. We are also grateful to Yuhan Sun's comments on the draft of this paper. \zjnote{In particular, we thank Jungsoo Kang's comments, which are formulated as Remark \ref{rmk-Kang-1} and Remark \ref{rmk-Kang-2}, as well as a correction about Example \ref{ex-HW}.} The main conclusion of this paper was presented by the second author in the conference ``Hamiltonian dynamics and Celestial Mechanics'' at SUSTech in Feburary, 2025, so we thank the invitation and hospitality of this conference. The second author is partially supported by National Key R\&D Program of China No.~2023YFA1010500, NSFC No.~12301081, and NSFC No.~12361141812. }

\section{Hamiltonian Floer theory and spectral invariant}\label{Ham-Floer}
\subsection{Liouville domains} \label{ssec-LD}
In this section, we briefly review the Hamiltonian Floer theory and spectral invariants on Liouville domains. We refer to \cite{Rit13}, \cite{BK22}, \cite{Mai24} and \cite{FZ24} for more details.

A Liouville domain $(W^{2n},\omega=d\theta)$ is a compact exact symplectic manifold with boundary, where the Liouville vector field $V_\theta$, defined by $\iota_{V_\theta}\omega=\theta$, points outwards along $\partial W$. We always assume $c_1(TW)|_{\pi_2(W)}=0$. The restriction of $\theta$ to the boundary, denoted by  $\alpha\coloneqq\theta|_{\partial W}$, is a contact form on $\partial W$. We can extend $W$ to a complete manifold using the flow $\varphi_{V_\theta}^t$ generated by $V_\theta$: 
\begin{equation} \label{completion}
\widehat{W}\coloneqq W\bigcup (1,\infty)\times\partial W.
\end{equation}
Here, we identify $(1,\infty)\times\partial W$ with $\bigcup_{r \geq 1} \varphi^{\log r}_{V_\theta}(\partial W)$. Then the one-form $\theta$ is extended to $\widehat{W}$ by defining $\theta=r\alpha$ for $r\in[1,\infty)$. Consequently,  $(\widehat{W},d\theta)$ becomes a non-compact exact symplectic manifold. The skeleton of $(W, d\theta)$ is defined by
\begin{equation} \label{dfn-skeleton}
{\rm Sk}(W)=\bigcap_{t>0}\varphi^{-t}_{V_{\theta}}(W).
\end{equation}
\jznote{Denote by $\{r=r_0\}$} the region $\varphi_{V_{\theta}}^{\log r_0}(\partial W)$ and \jznote{by $\{r\leq r_0\}$ the} region $\varphi_{V_\theta}^{\log r_0}(W)$ for \jznote{any} $0<r_0<1$. 

Let $H\in C^\infty([0,1]\times\widehat{W})$ be a Hamiltonian on $\widehat{W}$ and $\cP(H)$ be the set of contractible $1$-periodic  orbits of $\varphi^t_H$. A Hamiltonian orbit $x(t)=\varphi^t_H(x(0))$ is said to be non-degenerate if $\det((\varphi_H^1)_*|_{T_{x(0)}W}-\mathds{1})\ne 0$. We call a Hamiltonian $H$ non-degenerate if all elements in $\cP(H)$ are non-degenerate. 

Denote \jznote{by $\cL\widehat{W}$ the} space of contractible orbits in $\widehat{W}$. \jznote{The Hamiltonian action functional $\cA_H:\cL\widehat{W}\ra \bR$ for $H$ is defined by}
\[
 \cA_{H}(x)=-\int_0^1x^*\theta+\int_0^1H(x(t)).
\]
\jznote{It is a standard fact that the} set of critical points of $\cA_H$ is just $\cP(H)$.

Recall that the Reeb vector field $R_\alpha$ of $(\partial W,\alpha)$ is determined by
$$
d\alpha(R_\alpha,\cdot)=0\quad\text{ and }\quad \alpha(R_\alpha)=1.
$$
A closed Reeb orbit in $(\partial W,\alpha)$ is a closed orbit $\gamma:[0,T]\ra \partial W$ satisfying $\dot{\gamma}(t)=R_\alpha(\gamma(t))$ and $\gamma(0)=\gamma(T)$. \jznote{The} set of periods of all closed  Reeb orbits is denoted by $\spec(\partial W,\alpha)$, which is known to be a closed nowhere dense set in $(0,+\infty)$.  Moreover, if we take $H=h(r)$ on $(r_0,\infty)\times \partial W$ for some $r_0 \geq 1$, then  the Hamiltonian vector field has the form $X_H(r, x)=h'(r)R_\alpha(x)$ on $(r_0,\infty)\times \partial W$. Any Hamiltonian orbit $x$ of $H$ in $(r_0,\infty)\times \partial W$ is restricted to $\{r_1\}\times\partial W$ for some $r_1>r_0$ and corresponds to the Reeb orbit $\gamma(t)=x(t/T)$ with period $T=|h'(r_1)|$.  If the orbit $x$ lies in $(1,\infty)\times\partial W$ and $H=h(r)$ on $(1,\infty)\times \partial W$, then
\begin{equation}\label{eq-action}
\cA_H(x)=-rh'(r)+h(r)
\end{equation}
that is, the action of $x$ is equal to the $y$-intercept of the tangent line of the function $y=h(r)$ at $r$. \jznote{We} call $H\in C^\infty([0, 1]\times \widehat{W} )$ an {\it admissible} Hamiltonian if $H$ has the form
\begin{equation} \label{h-mu}
H(t,r,x)=\mu_Hr+a\quad\text{ on }[1,\infty)\times \partial W,
\end{equation}
where $\mu_H\notin  \spec(\partial W,\alpha)$ is a positive number (called the slope of $H$). \jznote{Denote} by $\cH$ the set of admissible Hamiltonians. 

\jznote{The Floer complex for the admissible \qfnote{non-degenerate} $H$ is defined as follows,}
\[
\left(\CF^*(H)=\bigoplus_{x\in\cP(H)}\bZ_2\cdot x,d_{H,J}\right).
\]
The degree of $x\in\cP(H)$ is defined by 
\[
|x|=\dfrac{1}{2}\dim W-\CZ(x)
\]
where $\CZ(x)$ denotes the Conley-Zehnder index of $x$.  We take this convention to ensure that $|x|$ equals the Morse index of $x$ when $x$ is a critical point of a $C^2$-small Morse function $W\ra\bR$ and the unit of \jznote{$\HF^*(W)$} lies in
degree $0$  (see Remark~3.3 in \cite{Rit13}).
The Floer differential $d_{H,J}:\CF^*(H)\ra \CF^{*+1}(H)$ is defined by counting Floer cylinders and satisfies $d^2_{H,J}=0$.

More precisely, to calculate the Conley-Zehnder index of a Hamiltonian orbit $x$, we need to choose a symplectic trivialization $\tau\colon x^*TW\ra S^1\times\bR^{2n}$, where $\tau_t\colon T_{x(t)}W\ra \bR^{2n}$. \jznote{Then} define a path of symplectic matrices as follows:
\[
A^{x,\tau}(t)\coloneqq\tau_t\circ D\varphi_{X_H}^t(x(0))\circ\tau_0^{-1}
\]
and the Conley-Zehnder index $\mu_{CZ}(A^{x,\tau})$ gives the Conley-Zehnder index $\CZ(x)$. For degenerate orbit $x$, we can use Robbin-Salamon index to calculate it: $\CZ(x)=\mu_{RS}(A^{x,\tau})$ (see \cite{RS93}). This is well-defined since $2c_1(TW)|_{\pi_2(W)}=0$.

 For Reeb orbit $\gamma$ in $\partial W$, we can similarly define the Conley-Zehner index of $\gamma$  using the symplectic trivialization of $\gamma^*\xi$. More details can be found in Appendix~C in \cite{MR23}. In particular, if $H=h(r)$ depends only on $r$, \jznote{then} we can relate the Conley-Zehnder index of a one-periodic orbit $\tilde{\gamma}$ of $X_H$ in the level set $\{r=r_0\}$ to the Conley-Zehnder index of its corresponding $|h'(r_0)|$-periodic  Reeb orbit $\gamma$. Note that if $h'(r_0)<0$, then $\tilde{\gamma}$ and $\gamma$ are in the opposite direction. We can choose a basis of sections to trivialize $x^*\xi$, leading to a trivialization of $x^*TW\cong(\bR \partial_r\oplus\bR R_{\alpha})\oplus x^*\xi$, where $\xi$ is the contact structure given by $\xi=\ker(\alpha)$. Similar to the calculation in Section~3.1 of \cite{Ush14}, we have 
\[
(\varphi_{H}^t)\partial_r=\partial_r+th''(r_0)R_{\alpha} \quad\text{ and }\quad(\varphi_{H}^t)R_{\alpha}=R_\alpha.
\]
The family of symplectic matrices obtained from $\varphi_H^t$ can be identified with the family given by $\varphi_H^t|_{\xi}$, along with a shear of the form
\[
\left(\begin{matrix}
1& 0\\
t\cdot h''(r_0) &1\\
\end{matrix}\right)
\]
in the $(\partial_r, R_\alpha)$-plane, which contributes $\frac{1}{2}{\rm sgn}h''(r_0)$ to $\CZ(\tilde{\gamma})$ (see Propositon~4.9 in \cite{Gut14}). Thus, the Conley-Zehnder index of $\tilde{\gamma}$ is given by
\jznote{\begin{equation}\label{eq-index}
\CZ(\tilde{\gamma})={\rm sgn}\, h'(r_0)\cdot \CZ(\gamma)+\frac{1}{2}{\rm sgn}\, h''(r_0)
\end{equation}}
as stated in Remark~5.3 in \cite{MR23} (c.f. Section~3.1 in \cite{Ush14}).
The correction factor of $\frac{1}{2}$ \qfnote{cancels} out when considering the $S^1$-family of 1-orbits \jznote{$x(t)$ acted by rotations}. After  a small perturbation, the circle will split into two non-degenerate orbits with index $\CZ(\tilde{\gamma})\pm\frac{1}{2}$ (see Proposition~2.2 in \cite{CFHW96}). 

\jznote{Define a filtration (function)} $\ell_{H}$ on $\CF^*(H)$ by 
$$\ell_{H}\left(\sum_{i}a_i x_i\right)\coloneqq\min\{\cA_{H}(x_i)\mid a_i\ne 0\},$$
where $x_i\in \cP(H)$, then for any $x, y \in {\rm CF}^*(H)$ with \jznote{$d_{H,J} y = x$}, we have $\ell_{H}(y)\leq\ell_{H}(x)$. We define the action spectrum as 
\[
\spec(H)\coloneqq\{\ell_H(x)\mid x\in \CF^*(H)\}.
\] 
The Floer cohomology for $H$ is defined to be the quotient space $\ker(d_{H,J})/{\rm im}(d_{H,J})$ which is denoted by $\HF^*(H)$.  

\jznote{Now, consider} a partial order on the set of admissible Hamiltonians negative on $W$ \jznote{as follows:} $H\preceq K$ if $H(t,x)\leq K(t,x)$, for any \jznote{$(t,x)\in [0,1]\times \widehat W$}. Let $\{H_i\}_{i\in I}$ be a cofinal sequence with respect to $\preceq$. The symplectic cohomology of $W$ is defined by the direct limit
\[
\SH^*(W,\omega)\coloneqq\varinjlim_{H_i}\HF^*(H_i)
\]
with respect to the continuation maps \jznote{$\varphi^{H_j,H_i}\colon\HF^*(H_i)\ra\HF^*(H_j)$ for $i<j$}.

The spectral invariant for \jznote{$0\ne\alpha\in H^k(W)$} is defined by $$c(\alpha,H)\coloneqq\sup\{\ell_H(x)\in\bR\mid x\in\CF^k(H), dx=0, [x]=PSS_H(\alpha) \}$$
where $PSS_H$ is the PSS isomorphism (see  Section 15.2 in \cite{Rit13})
\[
PSS_H\colon H^*(W)\ra \HF^*(H).
\]

Now we \jznote{can} define the spectral invariant for compactly supported Hamiltonian diffeomorphisms \jznote{as follows}. Take \qfnote{
\[
C_{cc}(W)\coloneqq\{H+c\mid H\in C^{\infty}_{c}( W), c\in\bR\}
\]
and
\[
C_{cc}([0,1]\times W)\coloneqq\{H+c\mid H\in C^{\infty}_{c}([0,1]\times W), c\in\bR\}.
\]
}
The spectral invariant for \jznote{$0\ne\alpha\in H^k(W)$} and \qfnote{$H\in C_{cc}([0,1]\times W)$} is defined by $$c(\alpha,H)\coloneqq c(\alpha,\widehat{H})$$
where $\widehat{H}\in\cH$ is a regular Hamiltonian such that $\widehat{H}|_W$ is a $C^2$-small perturbation of $H$ and $\mu_H<\min\spec(\partial W,\alpha)$ (\jznote{where $\mu_H$ is the slope in (\ref{h-mu})}). 
This is well-defined since the regular cases are generic and the spectral invariant exhibits $C^0$-continuity by the following proposition, which is derived from Proposition~23 and Lemma~25 in \cite{Mai24} \jznote{(with a different sign convention)}. The following  proposition regarding the spectral invariant  holds not only \jznote{for Liouville} domains, \jznote{shared by} Lagrangian spectral invariants and spectral invariants on closed symplectic manifolds. 
\begin{prop}\label{prop-spectral-invariant}
Let $\alpha, \beta\in H^*(W)$ and \qfnote{$H,K\in C_{cc}([0,1]\times W)$}. Then
\begin{enumerate}
\item (Hamiltonian shift) $c(\alpha, H+\lambda)=c(\alpha, H)+\lambda$ for any $\lambda\in\bR$.
\item(Continuity) 
\[
\int_0^1\min_{x\in W}(H-K)dt\leq c(\alpha, H)-c(\alpha,K)\leq \int_0^1\max_{x\in W}(H-K)dt.
\]
\item (Spectrality) $c(\alpha, H)\in\spec(H)$.
\item (Anti-triangle inequality) $c(\alpha\cup\beta, H\sharp K)\geq c(\alpha, H)+c(\beta, K)$.
\end{enumerate}
\end{prop}

\jznote{The following proposition establishes an invariant property of spectral invariant in the setting of Liouville domains.}

\begin{prop}[Remark~5.1 in \cite{GT23}]\label{prop-extend}
Given an incompressible domain $U$ in a Liouville domain $W$ and a symplectic embedding $\Psi\colon (U,\omega)\ra (W',\omega')$ whose image is an incompressible domain in another Liouville domain $W'$, then for every Hamiltonian $F$ supported in $U$,
\[
c_W(\mW, F)=c_{W'}(\mathds{1}_{W'}, \Psi_*F)
\]
where $\Psi_*F \colon W' \times S^1\ra\bR$ is the extension by zero of $F\circ\Psi^{-1}$.
\end{prop}

\jznote{Using spectral invariant, one can define} the heaviness of subsets in a Liouville domain:
\begin{dfn}\label{dfn-heavy}
Let $(W, \omega)$ be a Liouville domain. For a compact set $K$ in $W$, we say $K$ is heavy in $(W, \omega)$ if $c(\mW,H)=0$ for any non-negative time-independent Hamiltonian \qfnote{$H\in C_{cc}(W)$} with $H|_K=0$.
\end{dfn}
\begin{remark}\label{rmk-spec-equi}
Entov-Polterovich defines heaviness for subsets in \jznote{a closed symplectic manifold} using symplectic quasi-states in Definition 1.3 of \cite{EP09}. Sun showes that heaviness defined by spectral invariants is equivalent to that defined by symplectic quasi-states in Lemma 2.4 of \cite{Sun24}. Here, we adopt the concept of heaviness as defined directly by spectral invariants for Liouville domains.
\end{remark}



Now we can state \jznote{\qfnote{a} proposition} on which the proof of Theorem~\ref{thm-A} is based on. \jznote{This is part of (3) in Proposition \ref{thm-Liouville}.}  

\begin{prop}\label{thm-partial-heavy}
Let $(W,\omega)$ be a Liouville domain. Suppose ${\rm Sk}(W)$ is heavy in $(W, \omega)$, then the set $\{r_0\}\times \partial W$ is heavy  in $(W, \omega)$ for any $0<r_0<1$.
\end{prop}

\jznote{The proof of this proposition will be given in Section \ref{ssec-proof-thm-partial-heavy}. Moreover,} we can use a family of  heavy hypersurfaces to show the existence of a quasi-flat, \jznote{which is part of (4) in Proposition \ref{thm-Liouville}.}

\begin{lemma}[cf.~Lemma~2.5 in \cite{Sun24}]\label{lem-flat-1}
Let $(W, \omega)$ be a Liouville domain and $Z$ be a smooth hypersurface such that it admits a neighborhood that is diffeomorphic to $I \times Z$ for some finite open interval $I \subset \bR$. Then if $\{r\}\times Z$ is heavy for any $r \in I$, then there exists a quasi-isometric embedding from $(C_{c}^\infty(I), d_{\infty})$ to $(\ham(W,\omega), d_{\gamma})$.
\end{lemma}

\jznote{The proof of this lemma will be given in Section \ref{ssec-proof-lem-flat-1}.}


\subsection{Admissible Lagrangians}

To define the wrapped Floer homology, we consider certain well-behaved Lagrangian submanifolds in Liouville domains. For more details, see \cite{Rit13,Gon24}.
\begin{dfn}\label{admissible-Lag}
    Let $L^n\subset(W,d\theta)$ be a connected exact Lagrangian submanifold with the Legendrian boundary $\partial L=L\cap\partial W$ such that the Liouville vector filed is tangent to $TL$ along the boundary. We call such $L$ an {\rm admissible Lagrangian} if $\theta|_L=dk_L$ for a function $k_L\in C^\infty(L,\bR)$ which vanishes in a neighborhood of the boundary $\partial L$, and the relative Chern class $c_1(W,L)\in H^2(W,L;\bZ)$ satisfies $2c_1(W,L)=0$.
\end{dfn}
Similarly to (\ref{completion}), one can extend admissible $L$ to an exact non-compact Lagrangian in $\widehat{W}$ by setting
$$\widehat{L}\coloneqq L\bigcup\phi^t_{V_\theta}(\partial L)$$
and setting $k_L=0$ on $\widehat{L}\backslash L$.
\begin{ex}\label{disk-conormal}
    Let $D_g^*N$ be the disk cotangent bundle of a closed manifold $(N,g)$. The disk conormal bundle
    $$\nu^*K=\left\{(q,p)\in D_g^*N|_K \, | \, p(v)=0, \forall v\in T_qK \right\}$$
    of a closed submanifold $K\subset N$ is an admissible Lagrangian with $k_L=0$. If we take $K=\{x\}$, then $F_x$ is an admissible Lagrangian (called the fiber at $x$) with $k_{F_x}=0$.
\end{ex}

A Hamiltonian chord of a smooth Hamiltonian function $H\in C^\infty([0,1]\times \widehat{W})$ is a chord with properties
$$
\dot{x}(t)=X_H(x(t)) \quad\hbox{and}\quad x(0), x(1)\in \widehat{L}.
$$
 Clearly,  Hamiltonian chords of $H$ correspond to the intersection points in $\phi^1_H(\widehat{L})\cap \widehat{L}$ where $\phi_H^1$ is the time one map of the flow of $X_H$. We denote $\cP(L;H)$ as the set of Hamiltonian chords. A Hamiltonian chord $x(t)=\phi_H^t(x(0))$ is said to be non-degenerate if $\phi_H^1(\widehat{L})$ intersects transversely with $\widehat{L}$ at $x(0).$ \qfnote{We call a Hamiltonian $H$ non-degenerate if all elements in $\cP(L;H)$ are non-degenerate. }

Denote $\cP(\widehat{L})$ as the space of paths which are from $\widehat{L}$ to $\widehat{L}$ and homotopic to the constant path relative to $\widehat{L}$. We can define the action functional $\cA_{L,H}\colon\cP(\widehat{L})\ra\bR$ for $H$:
\[
\cA_{L;H}(\gamma)=-\int_0^1 \gamma^*\theta+\int^1_0 H(\gamma(t))dt+k_{L}\big(\gamma(1)\big)-k_{L}\big(\gamma(0)\big)
\]
where $\theta=dk_{L}$ on $\widehat{L}$. It coincides with the definition of action functional defined in Section~4.3 in \cite{Rit13}. Therefore, the critical points of $\cA_{L;H}$ are the Hamiltonian chords.

Fixing an admissible Lagrangian $L_1\subset W$, a Reeb chord of period $T$ is a map $\gamma:[0,T]\to\partial W$ satisfying
$$
\dot{\gamma}(t)=R(\gamma(t))\quad\hbox{and}\quad \gamma(0), \gamma(T)\in\partial L.
$$
Then the set of periods of Reeb chords is denoted by $\spec(\partial L_0; \theta)$, which is known to be a closed nowhere dense subset in $(0,+\infty)$. We call $H\in C^\infty([0,1]\times \widehat{W})$ an {\it admissible Hamiltonian} if $H$ has the form
\begin{equation} \label{extension-H}
H(t,r,x)=\mu_Hr+a\quad\text{ on }[1,\infty)\times \partial W.
\end{equation}
Here, $\mu_H\notin  \spec(\partial L_0, \partial L_1,\theta)$ is a non-negative scalar known as the slope of $H$.  We denote by $\cH$ the set of all admissible Hamiltonians.  

Now we define the Floer complex for the admissible Lagrangian $L$ and admissible \qfnote{non-degenerate} $H$ by
\[
\left(\CW^*(L;H)=\bigoplus_{x\in\cP(L;H)}\bZ_2\cdot x,d_{H,J}\right).
\]
The Floer differential $d_{H,J}\colon \CW^*(L;H)\ra\CW^{*+1}(L;H)$ is defined by counting wrapped trajectories and satisfies $d^2_{H,J}=0$. For more details, we refer the reader to Section~9 in \cite{AS10}.

We can define a filtration $\ell_{L;H}$ on $\CW^*(L;H)$ by
\[
\ell_{L;H}\left(\sum_{i=1}^na_ix_i\right)\coloneqq\min\{\cA_{L;H}(x_i)\mid a_i\neq 0\}
\]
where $x_i\in\cP(L;H)$, then for any $x,y\in\CW^*(L;H)$ with $dy=x$, we have $\ell_{L;H}(y)\leq \ell_{L;H}(x)$.

We now define a partial order on the set of admissible Hamiltonians negative on $W$ by defining $H\preceq K$ if $H(t,x)\leq K(t,x)$, for any $(t,x)\in [0,1]\times W$. Let $\{H_i\}_{i\in I}$ be a cofinal sequence with respect to $\preceq$. The wrapped Floer cohomology of $L$ is defined by the direct limit
\[
\HW^*(L)\coloneqq\varinjlim_{H_i}\HW^*(L;H_i)
\]
with respect to the continuation maps \jznote{$
\varphi^{H_j,H_i}\colon\HW^*(L;H_i)\ra\HW^*(L;H_j)$ for $i<j$.}

The spectral invariant for $0\ne\alpha\in H^k(L)$ is defined by 
\[
\ell(\alpha,H)\coloneqq\sup\{\ell_{L;H}(x)\in\bR\mid x\in\CW^k(L;H), dx=0, [x]=PSS_H(\alpha) \}
\]
where $PSS_H$ is the PSS isomorphism (see  Section 3.2 in \cite{Gon24})
\[
PSS_H\colon H^*(L)\ra \HW^*(L;H).
\]

Now \jznote{define} the spectral invariant for compactly supported Hamiltonian diffeomorphisms. The spectral invariant for $0\ne\alpha\in H^k(L)$ and \qfnote{$H\in C_{cc}([0,1]\times W)$} is defined by 
\[
\ell(\alpha,H)\coloneqq \ell(\alpha,\widehat{H})
\]
where $\widehat{H}\in\cH$ is a regular Hamiltonian such that $\widehat{H}|_W$ is a $C^2$-small perturbation of $H$ and $\mu_H<\min\spec(\partial L,\alpha)$. 
This is well-defined since the regular cases are generic and spectral invariant satisfies properties similar to those in Proposition~\ref{prop-spectral-invariant}, which includes the  $C^0$-continuity. 

Now \jznote{define} the $L$-heaviness and $L$-superheaviness of the subset in the Liouville domain.
\begin{dfn}[\qfnote{cf. Definition~1.1 in \cite{Kaw18}}]\label{def-L-heavy} 
Let $L$ be an admissible Lagrangian in a Liouville domain $(W,\omega)$. For a compact subset $K$ in $W$,
\begin{enumerate}
\item it is $L$-heavy if $\ell(\mL,H)=0$ for any non-negative  time-independent Hamiltonian  \qfnote{$H\in C_{cc}(W)$} with $H|_K=0$;
\item it is $L$-superheavy if $\zeta(H)=0$ for any non-positive  time-independent Hamiltonian  \qfnote{$H\in C_{cc}(W)$} with $H|_K=0$. Here $\zeta(H)$ is the \jznote{following homogenization} \[
\zeta(H)\coloneqq\lim_{m\ra+\infty}\frac{\ell(\mL,mH)}{m}.
\]
\end{enumerate}
\end{dfn}
\begin{remark}\label{rmk-heavy-def}
There are various equivalent definition of heaviness and superheaviness. For instance, we can say $K$ is $L$-heavy if $\zeta(H)=0$ for any non-negative  time-independent Hamiltonian  $H$ with $H|_K=0$. The proof of equivalence is same as the proof of Lemma~2.4 in \cite{Sun24} and we omit it.
\end{remark}
\begin{prop} \label{prop-superheavy}
Let $L$ be an admissible Lagrangian in a Liouville domain $(W,\omega)$. If a compact subset $X$ is $L$-superheavy, then 
\begin{enumerate}
\item $X$ is $L$-heavy;
\item for any $L$-heavy subset $Y$, $X\cap Y\neq\emptyset$.
\end{enumerate}
\end{prop}
\begin{proof}
(1) Assume that $H|_K=0$ and $H\neq 0$. Then by  the triangle inequality for the spectral invariant $\ell$,
\[
\zeta(H)+\zeta(-H)=\lim_{m\ra+\infty}\frac{\ell(\mL,mH)+\ell(\mL,-mH)}{m}\leq\lim_{m\ra+\infty}\frac{\ell(\mL,0)}{m}=0.
\]
Since $K$ is $L$-superheavy, then $\zeta(H)=0$, which implies $\zeta(-H)\leq 0$. Note $\ell(\mL,-H)\geq 0$ by $-H\geq 0$, we have $\zeta(-H)\geq 0$. Then $\zeta(-H)=0$ for any $H$ with $H|_K=0$ and $H\neq 0$, which is the definition of $L$-heavy by Remark~\ref{rmk-heavy-def}.

(2) For any $L$-heavy $Y$, if $X\cap Y=\emptyset$, we can construct $H$ such that $H|_{Y}=0$, $H|_{X}=1$ and $0\leq H\leq 1$. Then $\zeta(H)=\zeta(H-1)+1=1$ by $X$ is $L$-superheavy. On the other hand, $\zeta(H)=0$ by $Y$ is $L$-heavy, and we get a contradiction.
\end{proof}

Now we can state \jznote{an} open analogues of Proposition~\ref{thm-partial-heavy} and Lemma~\ref{lem-flat-1}. 
\begin{prop}\label{thm-Lag-partial-heavy}
Let $(W,\omega)$ be a Liouville domain. Suppose ${\rm Sk}(W)$ is $L$-heavy in $(W, \omega)$, then the set $\{r_0\}\times \partial W$ is $L$-heavy  in $(W, \omega)$ for any $0<r_0<1$.
\end{prop}

\begin{lemma}\label{lem-Lag-flat}
Let $(W, \omega)$ be a Liouville domain and $Z$ be a smooth hypersurface such that it admits a neighborhood that is diffeomorphic to $I \times Z$ for some finite open interval $I \subset \bR$. Then if $\{r\}\times Z$ is $L$-heavy for any $r \in I$, then there exists a quasi-isometric embedding from $(C_{c}^\infty(I), d_{\infty})$ to $(\cO(L), d_{\gamma})$.
\end{lemma}
The proof of Proposition~\ref{thm-Lag-partial-heavy} \qfnote{and} \jznote{Lemma~\ref{lem-flat-1}} follow the same argument as that of Proposition~\ref{thm-partial-heavy} \jznote{and Lemma~\ref{lem-Lag-flat}, repsectively, so we omit the proof}.

\subsection{Closed manifolds}
Now we briefly recall the Hamiltonian Floer theory on closed manifolds. We refer \cite{Oh15} for more general theory of Hamiltonian Floer homology.

Let $(M,\omega)$ be a closed weakly monotone symplectic manifold. Let $H\colon S^1\times M\ra\bR$ be a time-dependent non-degenerate Hamiltonian on $M$ and let $\cP(H)$ be the set of contractible $1$-periodic orbits of $\varphi_{H}^t$. For an orbit $\gamma\in\cP(H)$, a capping of $\gamma$ is a smooth map $u\colon \bD^2\ra M$ satisfies $u(e^{2\pi it})=\gamma(t)$ for all $t$. We say $(\gamma, u)$ is equivalent to $(\gamma, v)$ if 
\[
\int_{\bD^2}u^*\omega=\int_{\bD^2}v^*\omega\,\,\,\,\text{ and }\,\,\,\, c_1(TM)([\overline{u}\sharp v])=0
\]
where $[\overline{u}\sharp v]$ is the homology class of the sphere obtained from gluing the disks $u$ and $v$ along their boundaries, with the orientation on $u$ reversed. The action of the capped orbit $[\gamma, u]$ is given by
\[
\cA_H([\gamma, u])=-\int_{\bD^2}u^*\omega+\int_0^1H(t,\gamma(t))dt
\]
and the degree of $[\gamma,u]$ is defined by
\[
\deg([\gamma,u])=\frac{\dim M}{2}-\CZ([\gamma,u])
\]
where $\CZ([\gamma, u])$ is the Conley-Zehnder index of this capped orbit. Given $A\in\pi_2(M)$ and a capped orbit $[\gamma,u]$, let $[\gamma, A\sharp u]$ be the result of gluing a representative of $A$ to $u$, we have
\[
\cA_{H}([\gamma, A\sharp u])=\cA_{H}([\gamma, u])-\omega(A)\,\,\text{ and }\,\,\deg([\gamma,A\sharp u])=\deg([\gamma,u])-2c_1(TM)(A).
\]
 The degree-$k$ part of the Floer complex is defined by
\[
\CF^k(H)\coloneqq\left.\left\{\sum_{i}a_i[\gamma_i, u_i]\,\right|\, a_i\in\bZ_2, \deg([\gamma_i, u_i])=k, \lim_{i\ra+\infty}\cA_H([\gamma_i,u_i])=+\infty\right\}
\]
The differential $d\colon\CF^*(H)\ra\CF^{*+1}(H)$ is defined by counting Floer cylinders and satisfies $d^2=0$.
We can define a filtration $\ell_H$ on $\CF^*(H)$ by
\[
\ell\left(\sum_i a_i[\gamma_i, u_i]\right)\coloneqq\min\{\cA_H([\gamma_i, u_i])\mid a_i\neq 0\}
\]
then for any $x\in\CF^*(H)$, we have $\ell_H(dx)\geq \ell_H(x)$. The Floer complex is filtered by
\[
\CF^*_{\geq t}(H)\coloneqq\{x\in\CF^*(H)\mid \ell_H(x)\geq t\}.
\]
The spectral invariant for $0\ne\alpha\in H^k(M)$ is defined by $$c(\alpha,H)\coloneqq\sup\{\ell_H(x)\in\bR\mid x\in\CF^k(H), dx=0, [x]=PSS_H(\alpha) \}$$
where $PSS_H$ is the PSS isomorphism (see   \cite{PSS96})
\[
PSS_H\colon \QH^*(M,\omega)\ra \HF^*(H).
\]
The spectral invariant on closed symplectic manifolds satisfies properties similar to those in Proposition~\ref{prop-spectral-invariant}.

For general closed symplectic manifold, the homotopy invariance of the spectral invariant only holds for any $H$ and $K$  generating the same element in $\widetilde{\ham}(M)$. We can define the spectral pseudo-norm by
\[
\gamma(H)\coloneqq-c(\mM, H)-c(\mM,\overline{H}).
\]
By taking infimum of all possible Hamiltonian, we obtain the spectral norm by
\[
\gamma(\varphi)\coloneqq\inf_{\varphi_H^1=\varphi}\gamma(H).
\]
McDuff and Kawamoto found in some condition, $\gamma$ does not depend on the choice of $H$, see Example~\ref{ex-descend}.

Similarly, we will define the heaviness and use a family of  heavy hypersurfaces to show the quasi-flat. Here we take the equivalent definition of heaviness and superheaviness by Lemma~2.4 in \cite{Sun24}.
\begin{dfn}\label{def-heavy}
For a compact set $K$ in $M$, 
\begin{enumerate}
\item it is heavy if $c(\mM,H)=0$ for any non-negative  time-independent Hamiltonian $H$ with $H|_K=0$.
\item it is superheavy if $\zeta(H)=0$ for any non-positive  time-independent  Hamiltonian $H$ on $M$ with $H|_K=0$. Here $\zeta(H)$ is the  \jznote{following homogenization}
\[
\zeta(H)\coloneqq\lim_{m\ra+\infty}\frac{c(\mM,mH)}{m}.
\]
\end{enumerate}
\end{dfn}
\jznote{With a similar proof of Lemma~\ref{lem-flat-1},} we can use a family of  heavy hypersurfaces to show the existence of a quasi-flat:
\begin{lemma}[cf.~Lemma~2.5 in \cite{Sun24}]\label{lem-flat-2}
Let $Z$ be a smooth hypersurface and a neighborhood of $Z$ is diffeomorphic to $(0,1)\times Z$. If the pseudo-norm $\widetilde{\gamma}$ descend to the norm $\gamma$  and $\{r=r_0\}=\{r_0\}\times Z$ is heavy for any $0<r_0<1$, then there exists a quasi-isometric embedding from $(C_{c}^\infty(I), d_{\infty})$ to $(\ham(M, \omega), d_{\gamma})$.
\end{lemma}

To prove Proposition~\ref{thm-closed}, we need another invariant from Floer theory. The boundary depth of Hamiltonian $H$ is defined by
\[
\beta(H)\coloneqq\inf\{a\in\bR\mid d(\CF^*_{\geq t})\subset \CF^*_{\geq t+a}\}.
\]
\jznote{One can} use boundary depth to estimate the low bound of the spectral norm:
\begin{prop}[Theorem~A in \cite{KS21}]\label{thm-beta-gamma}
 Let $(M,\omega)$ be a closed weakly monotone symplectic manifold, then for any Hamiltonian $H$, we have \jznote{$\beta(H)\leq \gamma(H)$}. 

\end{prop}

\section{Proofs}

\subsection{Proof of Proposition~\ref{thm-partial-heavy}} \label{ssec-proof-thm-partial-heavy}
Our proof is motivated by Theorem~3.2 and \jznote{Theorem} 3.4 in \cite{Sun24}. 

For any non-negative smooth Hamiltonian $H$ vanishes on $\{r=r_0\}$ and any small $\delta$>0, we can find a sufficiently large smooth $F$ such that $F\geq H$ and $c(\mW,F)<\delta$. Since $H\leq F$ implies that $c(\mW,H)\leq c(\mW,F)$, it follows that $c(\mW,H)< \delta$ holds for any $\delta>0$. Thus, we conclude that  $c(\mW,H)=0$.

Next, we need to show how to find $F$ such that $F\geq H$ and $c(\mW,F)<\delta$. Consider two smooth functions  $F$ and $F_0$  that satisfy the following conditions:
\begin{enumerate}
\item $H,F_0\leq F$ and $F=F_0=H$ on $\{r\geq r_0\}$.
\item $F$ and $F_0$  depend only on $r$ in the collar region $\{r_0-\varepsilon\leq r\leq r_0\}$ with $F'(r)< F'_0(r)<0$ on $\{r_0-\varepsilon\leq r\leq r_0-\varepsilon/100\}$ and $F'(r)=F'_0(r)<0$ on $\{r_0-\varepsilon/100\leq r\leq r_0\}$ for a sufficiently small $\varepsilon>0$.
\item $F''(r), F''_0(r)>0$ on $\{r_0-\varepsilon\leq r\leq r_0-2\varepsilon/3\}$ and $\{r_0-\varepsilon/3\leq r\leq r_0\}$.
\item $F'(r)=\ell_0$ on $\{r_0-\varepsilon/3\leq r \leq r_0-2\varepsilon/3\}$, where $\ell$ is a number such that  $\ell\notin -\spec{\alpha}$.
\item $F'_0(r)=\ell_0$ on $\{r_0-\varepsilon/3\leq r \leq r_0-2\varepsilon/3\}$, where $\ell_0$ is a number such that  $|\ell_0|<\min\spec{\alpha}$.
\item $F_0|_{\{ r\leq r_0-\varepsilon\}}=(F-s)|_{\{ r\leq r_0-\varepsilon\}}$  is a non-negative $C^2$-small Morse function and $F_0<\delta$ on $\{r\leq r_0-\varepsilon\}$. 
\end{enumerate}
Here we require $F'=F_0'$ on $\{r-\varepsilon/100\leq r\leq r_0\}$ to ensure the action of orbit in $\{r_0-\varepsilon\leq r\leq r_0\}$ must be larger than $\max F_0$.
A pictorial depiction of $F,H$ and $F_0$ is in Figure~\ref{fig-collar}.
\begin{figure}[ht]
	\centering
\includegraphics[scale=1]{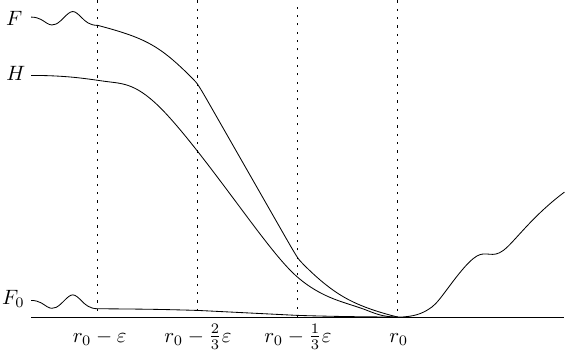}
\caption{Hamiltonian functions $F, H$ and $F_0$ in the collar region.}\label{fig-collar}
\end{figure}

We can find a non-negative function $\widetilde{F}_0$ that vanishes on $\{r\leq r_0\}$ such that $|F_0-\widetilde{F}_0|<\delta$.  Since ${\rm Sk}(W)$ is heavy, we have $c(\mW, \widetilde{F}_0)=0$. Thus $c(\mW, F_0)<c(\mW, \widetilde{F}_0)+ \delta=\delta$. The left thing is to show that $c(\mW,F)=c(\mW,F_0)$.

The one-periodic orbits of $F$ can be divided into three parts
\begin{enumerate}
\item Orbits in $\{r\geq r_0\}$.
\item Constant orbits in $\{r\leq r_0-\varepsilon\}$.
\item Non-constant orbits in the collar region, corresponding to Reeb orbits of $\alpha$.
\end{enumerate}

For constant orbits in $\{r\leq r_0-\varepsilon\}$, we will not perturb them as they are already non-degenerate. Let $F_t$ denote the non-degenerate Hamiltonian perturbed from $F$ and $F_{0,t}$ denote the non-degenerate Hamiltonian  perturbed from $F_0$. For any small $\delta_0>0$, we can suitably chose the $C^2$-small perturbation such that 
\begin{enumerate}
\item $F_t=F_{0,t}$ on $\{r\geq r_0-\varepsilon/100\}$. 
\item $|F_t-F|, |F_{0,t}-F_0|\leq \delta_0$.
\item $F_t, F_{0,t}\geq 0$.
\end{enumerate} 
The action of any  contractible orbit $\gamma$ lies in the collar region $\{r_0-\varepsilon\leq r\leq r_0-\varepsilon/100\}$ is large than $\max F_{0,t}|_{\{r\leq r_0\}}$ by \eqref{eq-action}. Thus, we obtain 
\[
\spec(F_{t})\cap(-\infty, \max F_{0,t}|_{\{r\leq r_0\}}+\varepsilon_0)=\spec(F_{0,t})\cap(-\infty, \max F_{0,t}|_{\{r\leq r_0\}}+\varepsilon_0)
\]
for some small enough positive $\varepsilon_0, \delta_0$. Since $F_{t}$ is non-degenerate and $\omega|_{\pi_2(W)}=0$, $\spec(F_{t})$ is discrete. We can fix a small $\varepsilon'>0$ such that
\[
(c(\mW, F_{0,t})-\varepsilon', c(\mW, F_{0,t})+\varepsilon')\cap\spec(F_{t})\cap (-\infty, \max F_{0,t}|_{\{r\leq r_0\}}+\varepsilon')=\{c(\mW, F_{0,t})\}.
\]
Similarly, we can construct a sequence of functions $\{F_{i,t}\}_{i=1}^n$ such that
\begin{enumerate}
\item   $F_{n,t}=F_t$.
\item $(c(\mW, F_{0,t})-\varepsilon', c(\mW, F_{0,t})+\varepsilon')\cap\spec(F_{i,t})\cap (-\infty, \max F_{0,t}|_{\{r\leq r_0\}}+\varepsilon')=\{c(\mW, F_{0,t})\}$.
\item  $\int_{S^1}\max_M (F_{i+1,t}-F_{i,t})<\frac{\varepsilon'}{2}$.
\end{enumerate}
By the \jznote{$C^0$-continuity} of the spectral invariant, we have 
\[
|c(\mW, F_{i+1, t})-c(\mW, F_{i, t})|<\frac{\varepsilon'}{2}.
\]
From (2), it follows that $c(\mW, F_{i+1,t})=c(\mW, F_{i, t})$.
By repeating this process inductively for $i$, we can show that $c(\mW, F_{i, t})$ is independent of $i$. Thus, we have 
\[
 c(\mW, F_{t})=c(\mW, F_{0,t}).
\]
Since $ |c(\mW,F)-c(\mW,F_t)|, |c(\mW,F_{0,t})-c(\mW,F_0)|\leq \delta_0$ by  the $C^0$-continuity of the spectral invariant, we have
$
|c(\mW,F)-c(\mW,F_0)|\leq 2\delta_0
$
for any $\delta_0$. Thus we conclude that $c(\mW,F)=c(\mW,F_0)$.

\subsection{Proof of Lemma~\ref{lem-flat-1}} \label{ssec-proof-lem-flat-1}

Without loss of generality, assume $I = (0,1) \subset \bR$. Given  a smooth compactly supported function $h\colon[0,1]\ra\bR$, construct a Hamiltonian $H\colon W\ra\bR$ such that
\begin{enumerate}
\item $H$ is supported in $\{0<r<1\}\times Z$. 
\item $H=h(r)$  depends only on $r$ on $\{0<r<1\}\times Z$.
\end{enumerate}
If $\min_{W}H= h(r_1)$ for some $0<r_1<1$,  define a new non-negative function $G=H-h(r_1)$ where $G|_{\{r=r_1\}}=0$. Since $\{r=r_1\}$ is heavy, then by Definition \ref{dfn-heavy}, we have $c(\mW, G)=0$. Thus, 
\[
c(\mW,H)=c(\mW,G+h(r_1))=c(\mW, G)+h(r_1)=h(r_1).
\] 
Similarly, if $\max_W H=h(r_2)$ for some $0<r_2<1$,  define a new non-negative function $F=h(r_2)-H$ where $F|_{\{r=r_2\}}=0$. Since $\{r=r_2\}$ is heavy, then by Definition \ref{dfn-heavy}, we have $c(\mW, F)=0$. Thus,
\[
c(\mW,-H)=c(\mW, F-h(r_2))=c(\mW, F)-h(r_2)=-h(r_2).
\]
Consider the following map 
\[
\Psi\colon C^\infty_c(I)\ra\ham(W,\omega) \quad \text{defined by} \quad h\mapsto\Psi(h) : = \varphi_H^1.
\]
For \qfnote{distinct} functions $h,h'\colon[0,1]\ra\bR$, the corresponding Hamiltonian $H$ and $H'$ satisfy $dH=fdH'$ for some function $f$ compactly supported on $(0,1)\times Z$ \qfnote{ and depending only} on $r$.  Since $\{H,H'\}=dH(X_{H'})=fdH'(X_{H'})=0$, we have $H$ and $H'$  Poisson commute, which implies that $\Psi(h)\circ\Psi(h')=\Psi(h+h')$. This leads to
\[
\gamma(\Psi(h))=-c(\mW,H)-c(\mW, -H)=\max_{[0,1]}h-\min_{[0,1]}h.
\]
Therefore, for any $f,g\in C^{\infty}_c(I)$, we obtain
\[
\|f-g\|_{\infty}\leq d_{\gamma}(\Psi(f),\Psi(g))=\gamma(\Psi(f-g))\leq 2\|f-g\|_{\infty}
\]
where the upper bound is given by the $C^0$-norm of $f-g$. Thus we complete the proof. 

\subsection{Proof of Proposition~\ref{prop-displace}} \label{proof-prop-displace} 
Suppose $K$ is heavy and can be displaced from itself by some $\psi\in\ham(W)$. Since $\psi(K)\cap K=\emptyset$ and $K$ is closed, we can find some neighborhood $U$ of $K$ such that $\psi(U)\cap U=\emptyset$. Take a Hamiltonian $H\colon W\ra\bR$ with $H\equiv 0$ on $K$, $H\equiv 1$ on $W\backslash U$ and $0\leq |H|\leq 1$.
Since $K$ is heavy, $c(\mW,sH)=0$ for any $s>0$. By Proposition~7.4 in \cite{FS07}, $\gamma(\varphi_{sH})\leq 2\gamma(\psi)$ for any $0<s<1$. By Lemma~B  in \cite{Mai24}, we have $c(\mW,s-sH)=0$ (due to different sign conventions). Then \jznote{we have the following computation,}
\begin{align*}
\gamma(\varphi_{sH})&=-c(\mW, sH)-c(\mW,-sH)\\
&=-c(\mW, sH-s)-c(\mW,s-sH)\\
&=s-c(\mW,s-sH)=s
\end{align*}
which contradicts the fact that $\gamma(\varphi_{sH})$ is bounded. Hence, $K$ can not be heavy.

\subsection{Proof of Proposition~\ref{thm-closed}} \label{ssec-proof-thm-closed}
The proof of Proposition~\ref{thm-closed} is based on the following estimate \jznote{on} the boundary depth:
\begin{prop}[Theorem~5.6 in \cite{Ush13}]\label{thm-bd}
Let $(M,\omega)$ be a closed symplectic manifold admits a nonconstant autonomous Hamiltonian $H\colon M\ra \bR$ such that the Hamiltonian vector field $X_H$ has no nonconstant contractible closed orbits. Then for any smooth  $f\colon \bR\ra\bR$ compactly supported on $[a,b]$, which are regular values of $H$, we have
\[
\beta(\varphi_{f\circ H}^1)\geq {\minmax} f-\min f
\]
where $\minmax f$ is defined by
\[
\minmax f\coloneqq\inf\{f(s)\mid s \text{ is a local maximum of }f\}.
\]
\end{prop}
Let $H\colon M\ra \bR$ be a nonconstant autonomous Hamiltonian  such that the Hamiltonian vector field $X_H$ has no nonconstant contractible closed orbits. For any regular value $a$ of $H$, consider  $f_{a}\colon \bR\ra\bR$ satisfies the following properties: 
\begin{enumerate}
\item ${\rm supp}(f_{a})=[a-\varepsilon, a+\varepsilon]$ for some small $\varepsilon>0$ such that $[a-\varepsilon, a+\varepsilon]$ are regular values of $H$.
\item $f_{a}'(r)<0$ if $a-\varepsilon\leq r<a$ and $f'(r)>0$ if $a< r\leq a+\varepsilon$.
\end{enumerate}
By Proposition~\ref{thm-bd}, we have
\[
\beta(f_a\circ H)\geq -f_a(a).
\]
By Proposition~\ref{thm-beta-gamma}, we have
\[
\gamma(f_a\circ H)\geq \beta(f_a\circ H)\geq -f_a(a).
\]
Since $c(\mM, \overline{f_a\circ H})=c(\mM,-f_a\circ H)\geq 0$ and $c(\mM,f_a\circ H)\geq\min_Mf_a\circ H= f_a(a)$, we have
\[
\gamma(f_a\circ H)=-c(\mM, f_a\circ H)-c(\mM,\overline{f_a\circ H})\leq -f_a(a).
\]
Thus $c(\mM,-f_a\circ H)= 0$ and $c(\mM,f_a\circ H)= f_a(a)$. For any non-negative Hamiltonian $G$ with $G|_{\{H=a\}}=0$, we can find $f_a$ such that $f_a\circ H-f_a(a)\geq G$. Thus $c(\mM, G)\leq c(\mM, f_a\circ H-f_a(a))=0$. So $c(\mM, G)=0$. We show that $\{H=a\}$ is heavy for any regular value $a$. Then by Lemma~\ref{lem-flat-2}, there exists a quasi-isometric embedding from $(C^\infty_c(I), d_0)$ to $(\ham(M,\omega), d_{\gamma})$.

\subsection{Proof of Proposition~\ref{thm-sun}}
The proof of Proposition is based on the following proposition from \cite{Sun24}:
\begin{prop}[\jznote{Theorem 3.2 and Theorem~3.4} in \cite{Sun24}]\label{thm-out-heavy}
\jznote{Under} the hypothesis of Proposition~\ref{thm-sun}, if there exists a Liouville domain $(K, d\lambda)$ such that
\begin{enumerate}
\item $(K, d\lambda)$ is symplectically embedded in $(W,\omega) $ with $\pi_1(K)\ra\pi_1(M)$ being injective.
\item $\alpha\coloneqq\lambda|_{\partial K}$ is a non-degenerate contact form.
\item $c_1(TK)|_{\pi_2(K)}=0$ and the Reeb orbits of $\lambda|_{\partial K}$ that are contractible in $K$ have Conley-Zehnder indices larger than $2-n$.
\item $K$ is heavy.
\end{enumerate}
Then $\partial K$ is heavy.
\end{prop}
\begin{remark}
When we consider the Hamiltonian orbit corresponding to the Reeb orbit of $\lambda|_{\partial K}$ and the $S^1$-symmetry of the Hamiltonian orbit, the Conley-Zehnder index will change by one, as shown in \jznote{equation}~\eqref{eq-index} and discussed \jznote{right} below it. \jznote{Therefore, the condition (3) in Proposition \ref{thm-out-heavy} can be relaxed to ``larger than $1-n$''}.
\end{remark}

An immediate example of Liouville domain $K$ \jznote{that} satisfies the condition in Proposition~\ref{thm-out-heavy} is \jznote{a unit cotangent bundle}. \jznote{Let us focus on the unit cotangent bundle of a }closed incompressible Lagrangian $L$ in $W$, \jznote{which serves as the $K$ in Proposition \ref{thm-out-heavy}. Moreover, in what follows, we will explain how the conditions in Proposition \ref{thm-out-heavy} can be satisfied for any $L$, argued according to its dimension and topology (especially in dimension 2).} 

\medskip

\noindent (1) If the dimension of $L$ is $\dim L\geq 3$, \jznote{then} the  Conley-Zehnder index of a contractible Reeb orbit in $S^*_gL$ can be calculated using the Morse index of its underlying geodesic, which is non-negative (see \cite{Liu05}). Since  non-degenerate Riemannian metrics are generic (see \cite{Ano82}), we can choose a neighborhood $D^*_gL$ of $L$ in $W$ for some non-degenerate Riemannian metric $g$. Then $D^*_gL$ satisfies the condition in Proposition~\ref{thm-out-heavy}. 
 
\medskip

\noindent (2) If $\dim L=2$ and $\chi(L)\leq 0$, then $L$ admits a metric $g$ with nonpositive curvature (see Theorem~1.2 in \cite{FM12}). In this case,  $(L,g)$ does not admit any non-trivial contractible closed geodesic. If there were a contractible geodesic $\gamma$ in $(L,g)$, it would bound a disk whose total curvature is $2\pi$ by Gauss-Bonnet Theorem,  contradicting our hypothesis. Therefore, $D^*_gL$ satisfies the condition in Proposition~\ref{thm-out-heavy}. 
 If $\dim L=1$, then there do not exist  any non-trivial contractible closed geodesics for any Riemannian metric $g$. Thus,  $D^*_gL$ also satisfies the condition in Proposition~\ref{thm-out-heavy}.

\medskip

\noindent (3) \jznote{If $\dim L = 2$ and $\chi(L) = 2$, then consider $(S^2, g)$ where $g$ is chosen so that geometrically $(S^2,g)$ }is a triaxial ellipsoid defined by
\[
a_1^2x^2+a_2^2y^2+a_3^2z^2=1
\] 
where $a_1>a_2>a_3$ are sufficiently near $1$. This surface has only three simple closed non-degenerate geodesics with Morse index $1, 2$ and $3$, see Theorem 2.1 in Chapter IX of  [22]. The Conley-Zehnder index $\CZ(\gamma)$ agrees with the Morse index of the corresponding geodesic on the ellipsoid, see Proposition~1.7.3 in \cite{EGH00} and Theorem~3.29 in \cite{Per22}. This Conley-Zehnder index is independent of the choice of the symplectic trivialization by Remark~1.4.3 in \cite{MvK22}. Thus $D^*_gS^2$  satisfies the condition in Proposition~\ref{thm-out-heavy}. 

\medskip

\noindent (4) \jznote{If $\dim L = 2$ and $\chi(L) = 1$, then consider \qfnote{($\bR P^2, g'$) which is the quotient of} the triaxial ellipsoid in (3) above \qfnote{under the reflection}. \qfnote{Note that the triaxial ellipsoid}} is invariant under the transformation $(x,y,z)\mapsto (-x,-y,-z)$, \qfnote{($\bR P^2, g'$) is well-defined.} 
 The contractible geodesics $\tilde{\gamma}$ on $(\bR P^2, g')$ correspond to the geodesics  $\gamma$ on $(S^2, g)$, and the Conley-Zehnder index of $\tilde{\gamma}$ coincides with the Conley-Zehnder index of $\gamma$. Therefore, $D^*_{g'}\bR P^2$ also  satisfies the condition in Proposition~\ref{thm-out-heavy}. 

\medskip

Now we can \jznote{proceed to the proof of} Proposition~\ref{thm-sun}. Let $L$ be an incompressible closed Lagrangian in $W$. We can pick a Weinstein neighborhood $U=U^*_gL$ of $L$ for some non-degenerate Riemannian metric $g$ \jznote{that} satisfies the condition in Proposition~\ref{thm-out-heavy}. Then apply Proposition~\ref{thm-out-heavy}, \jznote{and we know} $S^*_gL$ is heavy. Then by Lemma~\ref{lem-flat-2}, there exists a quasi-isometric embedding from $(C^\infty_c(I), d_0)$ to $(\ham(M,\omega), d_{\gamma})$.

\subsection{Proof of Theorem~\ref{thm-D} }
Let $K$ be a heavy but not \jznote{superheavy} subset in $M$. Then by Definition~\ref{def-heavy}, there exists a non-positive Hamiltonian $H$ on $M$ with $H|_{K}=0$ and $\zeta(H)\neq 0$. \qfnote{Since $H\leq 0$, we have $\zeta(H)<0$.  This implies that for sufficiently large $N$, whenever $m>N$,}
\[
c(\mM,mH)<\frac{m\zeta(H)}{2}.
\] 
\qfnote{By the heaviness of $K$, it follows that} $\zeta(\mM,\overline{mH})=0$. Therefore,
\[
\gamma(mH)=-c(\mM,mH)-c(\mM,\overline{mH})>-\frac{m\zeta(H)}{2}.
\]
For any $a\in\bR$, if $|a|\leq N+1$, then
\[
0\leq \gamma(aH)\leq \|aH\|_{\rm Hofer}\leq (N+1)\|H\|_{\rm Hofer}
\]
which implies 
\[
-\frac{\zeta(H)}{2}a-(N+1)\|H\|_{\rm Hofer}\leq\gamma(aH)\leq \|H\|_{\rm Hofer}\cdot a.
\]
If $|a|>N+1$, then
\[
-\frac{\zeta(H)}{2}a-\|H\|_{\rm Hofer}\leq\gamma(\lfloor a\rfloor H)-\gamma((\lfloor a\rfloor-a)H)\leq\gamma(aH)\leq \|H\|_{\rm Hofer}\cdot a.
\]
Consider the following map
\[
\Psi\colon\bR\ra\ham(M,\omega),\quad a\mapsto \Psi(a)=\varphi_{aH}^1.
\]
For any $a,b\in\bR$, we obtain
\[
-\frac{\zeta(H)}{2}(a-b)-(N+1)\|H\|_{\rm Hofer}\leq d_{\gamma}(\Psi(a),\Psi(b))=\gamma((a-b)H)\leq \|H\|_{\rm Hofer}\cdot (a-b).
\]
Thus, $\Psi$ is the wanted rank-$1$ quasi-flat.

\qfnote{Next, let $\{K_i\}_{i=0}^n$ be a family of heavy compact subsets.} Consider a sequence of functions $\{H_i\colon M\ra\bR\}_{i=0}^n$ such that
\begin{enumerate}
\item $H_i$ is supported in some neighourbood $U_i$ of $K_i$ and $H_i\leq 1$.
\item $H_i=1$ on $K_i$.
\item $U_{i}\cap U_j=\emptyset$ if $i\neq j$.
\end{enumerate}
We define the map
\[
F\colon \bR^n\ra\ham(M,\omega), a=(a_1,\cdots, a_n)\mapsto \sum_{i=0}^n a_iH_i, \text{ where }a_0=-\sum_{i=1}^na_i.
\]
If  $\min_{0\leq i\leq n} a_i=a_k\leq 0$, then $F(a)-a_k\geq 0$ and $\left.\left(F(a)-a_k\right)\right|_{K_k}=0$. As $K_k$ is heavy, then $c(\mM, F(a)-a_k)=0$. Thus,
\[
c(\mM, F(a))=a_k=\min_{0\leq i\leq n}a_i.
\]
\qfnote{Similarly,} if $\max_{0\leq i\leq n} a_i=a_k\geq 0$, then $a_k+\overline{F(a)}\geq 0$ and $\left.\left(a_k+\overline{F(a)}\right)\right|_{K_k}=0$. \qfnote{Again, by heaviness of $K_k$, we have} $c\left(\mM, a_k+\overline{F(a)}\right)=0$. Thus,
\[
c\left(\mM, \overline{F(a)}\right)=-a_k=-\max_{0\leq i\leq n}a_i.
\]
Combining the results above, we obtain
\[
\gamma(F(a))=-c(\mM,F(a))-c\left(\mM,\overline{F(a)}\right)=\max_{0\leq i\leq n}a_i-\min_{0\leq i\leq n}a_i\in[\|a\|_{\infty},2n\|a\|_{\infty}].
\]
Consider the following map
\[
\Phi\colon\bR\ra\ham(M,\omega),\quad a\mapsto \Phi(a)\coloneqq\varphi_{F(a)}^1.
\]
For any $a,b\in\bR$, we obtain
\[
\|a-b\|_{\infty}\leq d_{\gamma}(\Phi(a),\Phi(b))=\gamma(F(a-b))\leq  2n\|a-b\|_{\infty}.
\]
Thus, $\Phi$ is the wanted rank-$n$ quasi-flat.

\appendix
\section{Proof of Theorem~\ref{thm-flat}}
\label{app}

Recently, Trifa proves the following result:
\begin{theorem}[\cite{Tri24}]\label{thm-unbounded}
If $\underline{L}$ is a disjoint union of $k$ embedded smooth closed simple curves bounding discs of the same area $A$ in $(\bD, \omega_{\rm std})$, where $k\geq 2$, $A>\frac{1}{k+1}$ and ${\rm Area}(\bD)=1$, then the metric space $(\cO(\underline{L}), \delta_{\rm Hofer})$ is unbounded.
\end{theorem}
The proof of Theorem~\ref{thm-unbounded} involves constructing a non-trivial homogeneous quasi-morphism $r$ on $\ham(\bD,\omega_{\rm std})$ that is Lipschitz with respect to the Hofer distance and  vanishes on 
\[
\cS(\underline{L})\coloneqq\{\varphi\in\ham(\bD,\omega_{\rm std})\mid \varphi(\underline{L})=\underline{L}\}.
\]
Our proof also relies on the same strategy, together with some basic observations.

Before \jznote{giving} the proof of Theorem~\ref{thm-flat}, we need to recall the definition of a quasi-morphism.
\begin{dfn}
Let $G$ be a group. A map $q\colon G\ra\bR$ is called a {\it quasi-morphism} on $G$ if there exists a constant $D\geq 0$ such that for any $g,h\in G$
\[
|q(gh)-q(g)-q(h)|\leq D.
\]
The infimum of all $D$ such that this property is satisfied is called the defect of $q$. Moreover, a quasi-morphism $q$ is said to be homogeneous if for any $g\in G$ and $n\in\bZ$, we have
\[
q(g^n)=nq(g).
\]
In fact, given any quasi-morphism $q\colon G\ra\bR$, one can always obtain a homogenous quasi-morphism $\mu\colon G\ra\bR$ by
\[
\mu(g)=\lim_{n\ra\infty}\frac{q(g^n)}{n}.
\]
\end{dfn}
Let $I=\left(0,\frac{(k+1)A-1}{8}\right)$ and denote by $C^\infty_{c}$ the set of functions whose support is compactly contained in the interior of $I$. We will construct a map $\Psi\colon C^\infty_c(I)\ra\cO(\underline{L})$ such that 
\[
A\|f-g\|_{\infty}-B\leq \delta_{\rm CH}(\Psi(f), \Psi(g))\leq \|f-g\|_{\infty}.
\]
This establishes Theorem~\ref{thm-flat}.

 For $a\in(0,1]$, let $C_a\subset\bD$ be the circle centered at the origin that bounds a disc of area $a$. For any $f$, set
\[
\tilde{f}|_{C_{2A-2\eta_s}}=\begin{cases}
f(s), &0\leq s\leq \frac{(k+1)A-1}{8}\\
a(s)f(s-\frac{(k+1)A-1}{8}), & \frac{(k+1)A-1}{8}\leq s\leq \frac{(k+1)A-1}{4}\\
0, &\text{otherwise}
\end{cases}
\] 
where $\eta_s=\frac{(k+1)A-1-s}{2(k-1)}$ and $a\colon \bR\ra [-1,0]$ such that $\tilde{f}$ is a mean-zero function on $\bD$. 
\jznote{Define} $\Psi\colon C^\infty_c(I)\ra\cO(\underline{L})$ by the following expression
\[
\Psi(f)=\varphi_{\tilde{f}}^1(\underline{L}).
\]
We will construct a family of quasi-morphisms to estimate the lower bound of the distance between $\Psi(f)$ and $\Psi(g)$.
\begin{prop}\label{thm-quasi-morphism}
There exists a family of homogeneous  quasi-morphisms, denoted by $\{r_s\}_{s\in I}$, on $\ham(\bD,\omega_{\rm std})$ that satisfy  the following properties:
\begin{enumerate}
\item For any $\psi\in\ham(\bD,\omega_{\rm std})$ and $s\in I$, $|r_s(\psi)|\leq 2((k+1)A-1)\|\psi\|_{\rm Hofer}$.
\item For any $s\in I$, $r_s(\varphi_{\tilde{f}}^1)=\frac{(k+1)A-1}{2k}f(s)$.
\item For any $s\in I$, $r_s$ vanishes on $\cS(\underline{L})$.
\item Denote $D_s$ as the defect of $r_s$, then $D\coloneqq\sup_{s\in I}D_s<+\infty$.
\end{enumerate}
\end{prop}

\begin{proof}[Proof of Theorem \ref{thm-flat} (assuming Proposition \ref{thm-quasi-morphism})] 
Let $\psi\in\ham(M,\omega)$ be such that $\psi(\underline{L})=\varphi_{\tilde{f}}^1(\underline{L})$, then $\psi^{-1}\varphi_{\tilde{f}}^1\in\cS(\underline{L})$. Therefore $r_s(\psi^{-1}\varphi_{\tilde{f}}^1)=0$ for any $s\in I$ and
\begin{align*}
\|\psi\|_{\rm Hofer}\geq &\frac{1}{2((k+1)A-1)}\sup_{s\in I}|r_s(\psi)|\\
=&\frac{1}{2((k+1)A-1)}\sup_{s\in I}|r_s(\psi^{-1})|\\
= & \frac{1}{2((k+1)A-1)}\sup_{s\in I}|r_s(\psi^{-1})-r_s(\psi^{-1}\varphi_{\tilde{f}}^1)|\\
\geq &\frac{1}{2((k+1)A-1)} \sup_{s\in I}\{|r_s(\varphi_{\tilde{f}}^1)|-|r_s(\varphi_{\tilde{f}}^1)+r_s(\psi^{-1})-r_s(\psi^{-1}\varphi_{\tilde{f}}^1)|\}\\
\geq &\frac{1}{2((k+1)A-1)}\sup_{s\in I}\{|r_s(\varphi_{\tilde{f}}^1)|-D_s\}\\
\geq&\frac{1}{4k} \sup_{s\in I}|f(s)|-\frac{D}{2((k+1)A-1)}.
\end{align*}
 Thus for any $f,g\in  C^\infty_c(I)$,
 \begin{align*}
 \delta_{\rm CH}(\Phi(f), \Phi(g))=& \delta_{\rm CH}(\varphi_{\tilde{f}-\tilde{g}}(\underline{L}), \underline{L})\\
 \geq&\frac{1}{4k}\sup_{s\in I}|(f-g)(s)|-\frac{D}{2((k+1)A-1)}\\
 =&\frac{1}{4k}\|f-g\|_{\infty}-\frac{D}{2((k+1)A-1)}.
 \end{align*}
Moreover, for any $f,g\in  C^\infty_c(I)$,
\begin{align*}
 \delta_{\rm CH}(\Phi(f), \Phi(g))=& \delta_{\rm CH}(\varphi_{\tilde{f}-\tilde{g}}(\underline{L}), \underline{L})\\
&\leq \|\varphi_{\tilde{f}-\tilde{g}}^1\|_{\rm Hofer}\leq \|f-g\|_{\infty}.
\end{align*}
\jznote{Thus, we complete the proof.}
 \end{proof}

\begin{proof}[Proof of Proposition~\ref{thm-quasi-morphism}]
Consider a Lagrangian link $\underline{L}=\cup_{i=1}^k L_i$ consisting of $k$ pairwise-disjoint circles on $\bS$ and $\bS\backslash\underline{L}$ consists of planar domains $B_j$ with $1\leq j\leq s$. We denote by $\tau_j$ the number of boundary components of $B_j$ and $A_j$ its area. A Lagrangian link $\underline{L}$ is called a $\eta-$monotone if there exists $\eta\in\bR_{\geq 0}$ such that $2\eta(\tau_j-1)+A_j$ is independent of $j$ for $1\leq j\leq s$ (see \cite{CGHMSS22}, Definition~1.12). 

\medskip

\jznote{Recent work from \cite{CGHMSS22} successfully \qfnote{constructs} a quasi-morphism for certain Lagrangian links, sharing several useful properties of the classical quasi-morphism derived from (Lagrangian) Floer theory. }

 \begin{theorem}[Theorem~7.6 in \cite{CGHMSS22}]\label{thm-mu}
For any $\eta$-monotone Lagrangian link $\underline{L}$, the map $\mu_{\underline{L}}\colon\ham(\bS, \omega)\ra\bR$ is a homogeneous quasi-morphism  with the following properties:
 \begin{enumerate}
 \item (Hofer Lipschitz) $|\mu_{\underline{L}}(\varphi)-\mu_{\underline{L}}(\psi)|\leq d_{\rm Hofer}(\varphi, \psi)$.
 \item (Lagrangian control) If $H$ is mean-normalized and $H_t|_{L_i}=s_i(t)$ for each $i$, then
 \[
 \mu_{\underline{L}}(\varphi^1_H)=\frac{1}{k}\sum_{i=1}^k\int_0^1s_i(t)dt
 \]
 \item (Defect) The defect of $\mu_{\underline{L}}$ is bounded by $\frac{k+1}{k}(2\eta(\tau_j-1)+A_j)$.
 \end{enumerate}
 \end{theorem}
\jznote{Moreover, these quasi-morphisms} $\mu_{\underline{L}}$ are related as follows:
\begin{theorem}[Theorem~7.7 in \cite{CGHMSS22}]\label{thm-ind}
Suppose that $\underline{L}, \underline{L}'$ are $\eta-$monotone links in $\bS$ with the same number of components $k$, then the quasi-morphism $\mu_{\underline{L}}$ and $\mu_{\underline{L}'}$ coincide.
\end{theorem}

\jznote{Recall} the result from \cite{Tri24}. Consider a symplectic embedding $\Phi_s$ from the disc $(\bD,\omega_{\rm std})$ into a sphere $\bS_s$ of area $1+s$, where $s\in (0,(k+1)A-1]$. Then $\Phi_s(\underline{L})$ is an $\eta_s-$monotone Lagrangian link on $\bS_s$, where $\eta_s=\frac{(k+1)A-1-s}{2(k-1)}$. 
Therefore the pullback of $\mu_{\Phi_s(\underline{L})}$ by $\Phi_s$ descends to a homogeneous Hofer-Lipschitz quasi-morphism on $\ham(\bD,\omega_{\rm can})$, that we denote by $\mu_s$. Moreover, Morabito proves the following result in \cite{Mor23}:
 \begin{prop}[Proof of Lemma 3.6 in \cite{Mor23}]\label{thm-mor}
  For any $\varphi$ in $\cS(\underline{L},\bD,\omega)$, we have
  \[
\mu_{s_1}(\varphi)- \mu_{s_2}(\varphi) = \frac{\eta_{s_2}-\eta_{s_1}}{2k} {\rm lk}(b(\varphi)) 
\]
where ${\rm lk}$ is the linking number of a braid.
 \end{prop}

\jznote{Back to the proof of Proposition \ref{thm-quasi-morphism}, l}et $\delta\coloneqq\frac{(k+1)A-1}{4}$ and
 \[
 r_s\coloneqq 2\delta(\mu_{s}-\mu_{s+2\delta})-2\delta(\mu_{s+\delta}-\mu_{s+3\delta}).
 \] 
 Then for any $\psi\in\ham(\bD,\omega_{\rm std})$ and $s\in I$, $|r_s(\psi)|\leq 8\delta\|\psi\|_{\rm Hofer}$. Moreover, for any $s\in I$, $r_s$ vanishes on $\cS(\underline{L})$ by \jznote{Proposition}~\ref{thm-mor} and the defect of $r_s$ is bounded. 
  
By Theorem~\ref{thm-ind}, the homogenized quasi-morphisms defined using $\eta$-monotone links only depend on the number of components and the constant $\eta$. Thus we can compute $\mu_s$ by the concentric circles $\{C_{A+i(A-2\eta_s)}\}_{0\leq i\leq k-1}$. Since 
\[
{\rm supp}(\tilde{f})\cap\{C_{A+i(A-2\eta_s)}\}_{i\neq 1}={\rm supp}(\tilde{f})\cap\{C_{A+i(A-2\eta_{s+j\delta})}\}_{0\leq i\leq k-1, j=1,2,3}=\emptyset
\]
by Lagrangian control \jznote{property, (2) in Theorem~\ref{thm-mu}}, we have 
\[
r_s(\varphi_{\tilde{f}}^1)=\frac{2\delta}{k} f(s).
\]
\jznote{Thus, we complete} the proof.
\end{proof}

\section{Equivalence of heaviness}\label{app-2}

For any Liouville domain $(W,\omega)$, we define three types of heaviness. 

\medskip

\noindent \jznote{(1)} The first is an open analogue of heaviness, introduced by Entov-Polterovich in \cite{EP09} and defined via spectral invariants of Hamiltonian Floer homology; see Definition~\ref{dfn-heavy}.

\medskip

\noindent \jznote{(2)} The second, called $\SH$-heaviness, is an open analogue of $\SH$-visibility and defined by the symplectic cohomology with support; see Definition~\ref{dfn-SH-heavy}. The  $\SH$-visibility in the closed setting is introduced in \cite{TV23} and is defined by relative symplectic cohomology introduced in \cite{Var21}. The equivalence of $\SH$-visibility and heaviness in the closed setting has been  established  in \cite{MSV24}.
 
\medskip

\noindent \jznote{(3)} The third, called $\rho$-heaviness, is introduced in \cite{OS19}, which is defined via symplectic cohomology with respect to some subset; see Definition~\ref{dfn-rho-heavy}.

\medskip

\jznote{The main result in this appendix is the following equivalence:}

\begin{theorem}\label{thm-equi}
For any compact subset $K$ in a  Liouville domain $(W,\omega)$, the following are equivalent:
\begin{enumerate}
\item $K$ is heavy.
\item $K$ is $\SH$-heavy.
\item $K$ is $\rho$-heavy.
\end{enumerate}
\end{theorem}
\begin{proof}
The proof combines several results from this paper and other references.

(1)$\Rightarrow$(2): By Theorem~3.5 in \cite{MSV24}.

(2)$\Rightarrow$(3): By Proposition~2.1 in \cite{OS19}.

(3)$\Rightarrow$(1): By Proposition~\ref{prop-spec}.
\end{proof}
Furthermore, we can apply this equivalence to \jznote{give a proof of (1)$\Rightarrow$(2), (3)} in Proposition~\ref{thm-Liouville}:
\begin{prop}\label{prop-heaviness}
Let $(W,\omega)$ be a Liouville domain with $\SH^*(W,\omega)\neq 0$, then ${\rm Sk}(W)$ and $\{r_0\}\times\partial W$ for any $0<r_0<1$ are heavy.
\end{prop}

Now we introduce the notations and definitions mentioned before.
For subsets of a Liouville domain, there are two equivalent ways to define  symplectic cohomology. The first is the symplectic cohomology with respect to a subset, as introduced by \cite{OS19}, defined by  taking the direct limit over Hamiltonians. The second is the symplectic cohomology with compact support, an open analogue of relative symplectic cohomology introduced by \cite{Var21}, defined via the telescope construction.  The version employed in this article is close to the one in \cite{SUV25}.

\jznote{Consider a partial order} on $\cH$ by setting $H\preceq G$ whenever $H\leq G$. For any compact subset $K$ in $\widehat{W}$, \jznote{define} subsets of $\cH$ as follows:
 \[
 \cH(K)\coloneqq\{H\in\cH\mid H|_{S^1\times K}<0\}\quad\text{ and }\quad\cH_{\rm reg}(K)\coloneqq\cH(K)\cap\cH_{\rm reg}.
 \]
 For $F\in C_{cc}(S^1\times W)$, \jznote{define} the subsets of $\cH$ as follows:
 \[
 \cH(F)\coloneqq\{H\in\cH\mid H|_{S^1\times W}<F\} \quad\text{ and }\quad \cH_{\rm reg}(F)\coloneqq\cH(F)\cap\cH_{\rm reg}.
 \]
 For $-\infty\leq a<b\leq\infty$, \jznote{define} the symplectic cohomology of $\widehat{W}$ with respect to $K$ by
\[
\SH^*_{[a,b)}(\widehat{W};K)\coloneqq\varinjlim_{H\in\cH_{\rm reg}(K)}\HF^*_{[a,b)}(H).
\]
Furthermore, \jznote{define} 
\[
\SH^*_{\geq a}(\widehat{W};K)\coloneqq\varprojlim_{b\ra\infty}\varinjlim_{H\in\cH_{\rm reg}(K)}\HF^*_{[a,b)}(H)
\]
\[
\SH^*(\widehat{W};K)\coloneqq\varinjlim_{a\ra-\infty}\varprojlim_{b\ra\infty}\varinjlim_{H\in\cH_{\rm reg}(K)}\HF^*_{[a,b)}(H)
\]
and write
\[
\SH^*(W)\coloneqq\SH^*(\widehat{W}; W).
\]
We also define symplectic cohomology of $\widehat{W}$ with respect to $F$ as follows. For any $F\in C_{cc}(S^1\times W)$ and $a<b$, we define
\[
\SH^*_{[a,b)}(\widehat{W};F)\coloneqq\varinjlim_{H\in\cH_{\rm reg}(F)}HF^*_{[a,b)}(H).
\]
Furthermore, \jznote{define} 
\[
\SH^*_{\geq a}(\widehat{W};F)\coloneqq\varprojlim_{b\ra\infty}\varinjlim_{H\in\cH_{\rm reg}(F)}\HF^*_{[a,b)}(H)
\]
\[
\SH^*(\widehat{W};F)\coloneqq\varinjlim_{a\ra-\infty}\varprojlim_{b\ra\infty}\varinjlim_{H\in\cH_{\rm reg}(F)}\HF^*_{[a,b)}(H)\cong \SH^*(W).
\]

Now the spectral invariant of $F\in C_{cc}([0,1]\times W)$ associated with $e\in\SH^*(W)\backslash\{0\}$ \jznote{introduced in \cite{OS19} is defined as follows}, 
\begin{equation} \label{OS-spec-inv}
\rho(e;F)\coloneqq\sup\{a\in\bR\mid e\in\im(\SH^*_{\geq a}(\widehat{W};F)\ra\SH^*(W))\}.
\end{equation}
If $\SH^*(W)\neq 0$, we also use $\mW$ to denote the unit of $\SH^*(W)$, which is the image of the unit of $H^*(W)$ under the map $H^*(W)\ra\SH^*(W)$.
\begin{dfn}\label{dfn-rho-heavy}
A closed subset $K\subset W$ is called $\rho$-heavy if for any $H\in C_{cc}(W)$ with $H|_K=0$ and $H\geq 0$, we have $\rho(\mW; H)=0$.
\end{dfn}
\begin{remark}
The original definition of  $\rho$-heaviness in Definition 2.6 in \cite{OS19} relies on quasi-state, which is the homogenization of the spectral invariant in (\ref{OS-spec-inv}). As shown in Remark~\ref{rmk-spec-equi}, spectral invariants can be used directly to define this heaviness.
\end{remark}

To define the symplectic cohomology with support $K$, we need the following acceleration data:
\begin{enumerate}
\item A sequence of non-degenerate Hamiltonian functions Hn which monotonically approximate from below $\chi_K^\infty$ , the lower semi-continuous function which is zero on $K$ and positive infinity outside $K$.
\item  A monotone homotopy of Hamiltonian functions connecting $H_n$’s.
\item A suitable family of almost complex structures.
\end{enumerate}
Given these data, \jznote{consider the following} chain complex
\[
\tel^*(\cC)\coloneqq\bigoplus_{n=1}^\infty\left(\CF^*(H_n)\oplus\CF^*(H_n)[1]\right)
\]
by using the telescope construction. The differential $\delta$ is defined as follows, if $(x_n,x'_n)\in\CF^{k}(H_n)\oplus\CF^k(H_n)[1]$, then
\jznote{\[
\delta (x_n,x'_n)=(d_nx_n+x'_n,-d_nx'_n,-h_nx'_n)
\]}
where $d_n\colon\CF^k(H_n)\ra\CF^k(H_{n})$ is the Floer differential defined by counting the Floer cylinder and $h_n\colon\CF^k(H_n)\ra\CF^k(H_{n+1})$ is the continuation map. Note that the Floer complex $\CF^*(H)$ is equipped with an action filtration, we can take a degree-wise completion in the following way. We define a filtration $\ell_{\cC}$ on $\cC$ by taking the smallest action among its summand. Then \jznote{subcomplexes $\CF^*_{\geq a}(H_n)$ and $\tel^*_{\geq a}(\cC)$ contain} elements with filtration not less $a$ for any $a\in[-\infty,\infty)$. For $a<b$, we have the quotients
\[
\CF^*_{[a,b)}(H_n)\coloneqq\CF^*_{\geq a}(H_n)/\CF^*_{\geq b}(H_n),\quad \tel^*_{[a,b)}(\cC)\coloneqq\tel^*_{\geq a}(\cC)/\tel^*_{\geq b}(\cC).
\]
The degree-wise completion of the telescope is defined as 
\[
\widehat{\tel^k}(\cC)\coloneqq\varprojlim_{b\ra\infty}\tel^k_{(-\infty,b)}(\cC)
\]
Then the symplectic cohomology with support on $K$ is defined as
\[
\SH^*_{\widehat{W}}(K)=H\left(\widehat{\tel^*}(\cC),\delta\right).
\]
\qfnote{These two types of symplectic cohomology are equivalent.
\begin{prop}\label{prop-SH-equi}
 Let $K$ be a compact set in a Liouville domain $W$, then we have 
 \[
\SH^*(\widehat{W};K)\cong \SH^*_{\widehat{W}}(K).
\]
\end{prop}
}
\qfnote{We can use the symplectic cohomology with support to define the $\SH$-heaviness.}
\begin{dfn}\label{dfn-SH-heavy}
 \qfnote{Let $K$ be a compact set in a Liouville domain $W$} and let $\SH^*_{\widehat{W}}(K)$ be the symplectic cohomology with support on $K$, we say $K$ is $\SH$-heavy if $\SH^*_{\widehat{W}}(K)\neq 0$. 
\end{dfn}
By \qfnote{Proposition~\ref{prop-SH-equi} and }Proposition~2.1 in \cite{OS19}, the $\SH$-heaviness of $K$ is equivalent to the $\rho$-heaviness of $K$, as the image of $\mW$ under the map $\SH^*(W)\ra\SH^*(\widehat{W}; K)$ is the unit of $\SH^*(\widehat{W}; K)$. 

\medskip

The proof of Theorem~\ref{thm-equi} is based on the following \jznote{proposition}:
\begin{prop}\label{prop-spec}
Let $(W,\omega)$ be a Liouville domain with $\SH^*(W)\neq 0$, then for any $F\in C_{cc}(W)$, we have $\rho(\mW; F)\geq c(\mW;F)$.
\end{prop}
\begin{proof}
For any $F\in C_{cc}(W)$, \qfnote{we can find $\widehat{F}\in\cH_{\rm reg}(F)$ such that $\widehat{F}|_{W}$ is a $C^2$-small perturbation of $F$ and $\mu_{F}<\min\spec(\partial W,\alpha)$. Then} we consider the following commutative diagram
\[\begin{tikzcd}
	\displaystyle{\varprojlim_{b\ra\infty}H(\CF^*_{[a,b)}(\widehat{F}))} & \displaystyle{\varprojlim_{b\ra\infty}\varinjlim_{H\in\cH_{\rm reg}(F)}H(\CF^*_{[a,b)}(H))} \\
	\displaystyle{\varinjlim_{a\ra-\infty}\varprojlim_{b\ra\infty}H(\CF^*_{[a,b)}(\widehat{F}))} &\displaystyle {\varinjlim_{a\ra-\infty}\varprojlim_{b\ra\infty}\varinjlim_{H\in\cH_{\rm reg}(F)}H(\CF^*_{[a,b)}(H))}
	\arrow[from=1-1, to=1-2]
	\arrow[from=1-1, to=2-1]
	\arrow[from=1-2, to=2-2]
	\arrow[from=2-1, to=2-2]
\end{tikzcd}\]
which is the following commutative diagram
\[\begin{tikzcd}
	{\HF^*_{\geq a}(\widehat{F})} & {\SH^*_{\geq a}(\widehat{W};F)} \\
	{\HF^*(\widehat{F})} & {\SH^*(\widehat{W};F)}
	\arrow[from=1-1, to=1-2]
	\arrow[from=1-1, to=2-1]
	\arrow[from=1-2, to=2-2]
	\arrow[from=2-1, to=2-2]
\end{tikzcd}\]
\qfnote{If $\mW\in \HF^*_{\geq a}(\widehat{F})$, then $\mW\in\im(\SH^*_{\geq a}(\widehat{W};F)\ra\SH^*(\widehat{W};F)\cong\SH^*(W))$.}
By the definition of two spectral invariant, we have $\rho(\mW; F)\geq c(\mW;\widehat{F})=c(\mW;F)$.
\end{proof}

\begin{proof}[Proof of Proposition~\ref{prop-heaviness}]
By Proposition~A.2 in \cite{SUV25}, we have $\SH^*_{\widehat{W}}(r_0W)\cong \SH^*(r_0W)$, where $r_0W=\{r\leq r_0\}$ for any $0<r_0<1$. Since $\SH^*(W)\neq 0$ implies $\SH^*(r_0 W)\neq 0$, $r_0 W$ is heavy for any $0<r_0<1$ by the equivalence of heaviness and $\SH$-heaviness. 
\qfnote{
For any Hamiltonian $H\in C_{cc}(W)$ with $H|_{{\rm Sk}(W)}=0$ and $H\geq 0$, we can find a sequence of Hamiltonians $H_n\in C_{cc}(W)$ with $H_n|_{r_nW}=0$, $H_n\geq 0$ and $|H-H_n|\leq \varepsilon_n$. Here $r_n\ra 0$ and $\varepsilon\ra 0$ as $n\ra+\infty$. Since $r_nW$ is heavy, we have $c(\mW,H_n)=0$ for any $n$. Then by $C^0$-continuity, we have $c(\mW,H)=0$, which implies that
}
 \jznote{$\bigcap_{0<r_0<1} r_0 W={\rm Sk}(W)$} is heavy. Similarly, by Corollary~A.11 in \cite{SUV25}, we have $\SH^*_{\widehat{W}}(\{r_0\}\times\partial W)\cong \RFH_*(\{r_0\}\times\partial W)$. Note that $\RFH_*(\{r_0\}\times\partial W)\neq 0$ is equivalent to $\SH^*(r_0W)\neq 0$, \jznote{therefore} we have $\{r_0\}\times\partial W$ is heavy for any $0<r_0<1$.
\end{proof}

\begin{proof}[Proof of Proposition~\ref{prop-SH-equi}]
By Lemma~A.9 in \cite{SUV25}, we can choose  a family of increasing  Hamiltonians $\{H_{n,t}\}_{n=1}^\infty$ such that
\[
\varinjlim_{n\ra+\infty}\CF^*_{[a,b)}(H_{n,t}) \quad\text{ and } \quad H(\varinjlim_{n\ra+\infty}\CF^*_{[a,b)}(H_{n,t}))
\] 
are finite-dimensional for any $[a,b)$. Furthermore, we can use this family of Hamiltonians to calculate the symplectic cohomology with respect to $K$, as $\{H_{n,t}\}_{n=1}^\infty$ is cofinal in $\cH(K)$. Thus we have
\begin{align*}
\SH^*(\widehat{W};K)&=\varinjlim_{a\ra-\infty}\varprojlim_{b\ra+\infty}\varinjlim_{n\ra\infty}\HF^*_{[a,b)}(H_{n,t})\\
&=\varinjlim_{a\ra-\infty}\varprojlim_{b\ra+\infty}H(\varinjlim_{n\ra\infty}\CF^*_{[a,b)}(H_{n,t}))\\
&=\varinjlim_{a\ra-\infty}H(\varprojlim_{b\ra+\infty}\varinjlim_{n\ra\infty}\CF^*_{[a,b)}(H_{n,t}))\\
&=H(\varinjlim_{a\ra-\infty}\varprojlim_{b\ra+\infty}\varinjlim_{n\ra\infty}\CF^*_{[a,b)}(H_{n,t})).
\end{align*}
Here the  direct limit commutes with taking (co)homology, and for  finite-dimensional chain complexes, the inverse limit commutes with taking (co)homology (see Section~3.5 in \cite{Wei94}).
Meanwhile, we have
\[
\SH^*_{\widehat{W}}(K)=H(\varprojlim_{b\ra+\infty}\tel^*_{(-\infty,b)}(\cC))=H^*(\varprojlim_{b\ra+\infty}\varinjlim_{a\ra-\infty}\tel^*_{[a,b)}(\cC)).
\]
Then by Lemma A.10 in \cite{SUV25}, we have $\SH^*(\widehat{W};K)\cong \SH^*_{\widehat{W}}(K)$.
\end{proof}

\bibliographystyle{amsplain}
\bibliography{biblio}

\medskip

\end{document}